\documentclass[11pt]{article}
\usepackage[utf8]{inputenc}
\usepackage{url}
\usepackage{amsmath, amssymb, amsfonts}
\usepackage{amsthm}
\usepackage{color}
\usepackage{graphicx, graphics}
\usepackage{enumerate}
\usepackage{geometry}
\usepackage{pstricks}
\usepackage{authblk}
\usepackage{ulem}
\usepackage{rotating}
\usepackage{array}
\usepackage{upgreek}
\usepackage{subfig}

\makeatletter
\newcommand*\rel@kern[1]{\kern#1\dimexpr\macc@kerna}
\newcommand*\widebar[1]{%
  \begingroup
  \def\mathaccent##1##2{%
    \rel@kern{0.8}%
    \overline{\rel@kern{-0.8}\macc@nucleus\rel@kern{0.2}}%
    \rel@kern{-0.2}%
  }%
  \macc@depth\@ne
  \let\math@bgroup\@empty \let\math@egroup\macc@set@skewchar
  \mathsurround\z@ \frozen@everymath{\mathgroup\macc@group\relax}%
  \macc@set@skewchar\relax
  \let\mathaccentV\macc@nested@a
  \macc@nested@a\relax111{#1}%
  \endgroup
}

\newcommand{\ben}{\vspace{0mm}\begin{equation}}
\newcommand{\een}{\vspace{0mm}\end{equation}}
\newcommand{\be}{\vspace{0mm}\begin{equation*}}
\newcommand{\ee}{\vspace{0mm}\end{equation*}}
\newcommand{\bea}{\vspace{0mm}\begin{equation*}\begin{aligned}}
\newcommand{\eea}{\vspace{0mm}\end{aligned}\end{equation*}}
\newcommand{\bean}{\vspace{0mm}\begin{equation}\begin{aligned}}
\newcommand{\eean}{\vspace{0mm}\end{aligned}\end{equation}}

\newcommand{\Acal}{\mathcal{A}}

\newcommand{\Ecal}{\mathcal{E}}

\newcommand{\B}{\mathcal{B}}

\newcommand{\E}{\mathbb{E}}
\newcommand{\Pro}{\mathbb{P}}
\newcommand{\N}{\mathbb{N}}
\newcommand{\R}{\mathbb{R}}

\newcommand{\un}{\mathbf{1}}

\newcommand{\pib}{\widebar{\pi}}

\newcommand{\norm}[1]{||#1||}

\newcommand{\Ze}{Z^{\varepsilon}}
\renewcommand{\ne}{N^{\varepsilon}}
\newcommand{\he}{H^{\varepsilon}}

\newcommand{\pie}{\pi^{\varepsilon}}

\newcommand{\av}{\widebar{N}}

\newtheorem{theorem}{Theorem}[section]
\newtheorem*{theorem*}{Theorem}
\newtheorem{thmx}{Theorem}

\newtheorem{cor}[theorem]{Corollary}
\newtheorem{rem}[theorem]{Remark}
\newtheorem*{rem*}{Remark}
\newtheorem{prop}[theorem]{Proposition}

\begin{document} 
\title{A piecewise deterministic model for a prey-predator community}
\author[1]{Manon Costa \footnote{Institut de Math\'ematiques de Toulouse.  CNRS UMR 5219, \\
Universit\'e Paul Sabatier, 118 route de Narbonne, 31062 Toulouse cedex 09, France\\ Email: manon.costa@math.univ-toulouse.fr}}
\date{1/03/2016}
\maketitle
\begin{abstract}
We are interested in prey-predator communities where the predator population evolves much faster than the prey's (e.g. insect-tree communities). We introduce a piecewise deterministic model for these prey-predator communities
that arises as a limit of a microscopic model when the number of predators goes to infinity. We
prove that the process has a unique invariant probability measure and that it is exponentially ergodic. Further on, we rescale the predator dynamics in order to model predators of smaller size. This slow-fast system converges to a community process in which the prey dynamics is averaged on the predator equilibria. This averaged process admits an invariant probability measure which can be computed explicitly. We use numerical simulations to study the convergence of the invariant probability measures of the rescaled processes.
\end{abstract}



\subsection*{ERRATUM - December 2025 } The statement and the proof of Theorem 2.2 concerning the ergodicity of a piecewise deterministic process $(Z_t)_{t\ge0}$ defined on $\N^*\times[0,\infty[$ is wrong stated as it is. A corrigendum can be found here (\url{https://hal.science/hal-04839491v1}) and is currently under review. \\

\section{Introduction}
Prey-predator communities represent elementary blocks of complex ecological communities and their dynamics has been widely studied. The coupled dynamics of the prey and predator populations is often described as a coupled system of differential equations. The most famous of them was introduced by Lotka \cite{lotka} and Volterra \cite{volterra} in the 1920's.
There exist also stochastic models for these prey-predator communities as coupled birth and death processes (see Costa and al. \cite{costa2014}) or as stochastic perturbations of deterministic systems (e.g. Rudnicki and  Pich\'or \cite{rudnicki2007influence}). 

In this paper we are interested in prey-predator communities in which the predator dynamics is much faster than the prey one. Such communities are common in the wild, especially if we consider the interaction between trees and insects (see Robinson and al. \cite{Umea} for the study of Aspen canopy and its arthropod community or Ludwig and al. \cite{ludwig1978qualitative} for the interaction between spruce budworm and the forest). 
In these communities, the number of predators is much larger than the prey number and the predator mass is smaller than the prey one. As a consequence, the reproduction and death events will be more frequent in the predator population  than in the prey population. Such slow-fast scales have been studied in some Lotka-Volterra systems, mainly in the case of periodic solutions (e.g. \cite{RINALDI1992287}). In the following, we study two successive scaling (large predator population and small predator mass) of a microscopic model of the prey-predator community and introduce new stochastic processes for slow-fast prey-predator systems which corresponds to these scalings limits. In particular one interest of our work is that these processes never assume that the prey population size is infinite, as it is the case for models based on ordinary differential equations.

We introduce a hybrid model for the demographic dynamics of the community where the prey population evolves according to a birth and death process while the dynamics of predators is driven by a differential equation. The community has a deterministic dynamics between the jumps of the prey population. 
This piecewise deterministic process arises as limit of a prey-predator birth and death process, when the number of predators tends to infinity while the prey number remains finite.
Such piecewise deterministic Markov processes (PDMP in short) where introduced by Davis in 1984 \cite{davis1984piecewise,davis1993markov}. They are used to model different biological phenomena. As an example, the dynamics of chemostats has been described 
as a piecewise deterministic model
 \cite{crump1979some,campillo2011stochastic, collet2013stochastic,genadot2014multi}. Chemostats, in which bacteria evolve in an environment with controlled resources, correspond to the opposite setting where the prey population (the resources) evolves faster than their predators (the bacteria). Other examples can be found in neuroscience, to model the dynamics of electric potentials in neurons (see Austin \cite{austin2008emergence}) or in molecular biology, where piecewise deterministic processes appear as various limits of individual based models of gene regulatory networks when the different interactions happen on different time scales (see Crudu and al. \cite{crudu2012convergence}).

In this paper we study the long time behavior of the prey-predator community process. A vast literature concerns the long time behavior of continuous time Markov processes. In the setting of piecewise deterministic processes,
general results have been obtained by Dufour and Costa on the relationships between the stationary behavior of the process and a sampled chain (see \cite{dufour1999stability,costa2008stability}). We focus on the theory of Harris-recurrent processes (see Meyn and Tweedie \cite{meyntweedie2,meyntweedie3} and references therein) that relies on Foster-Lyapunov inequalities. These inequalities satisfied by the infinitesimal generator of the process, ensure that the populations do not explode in some sense. Combined with irreducibility properties, they ensure the existence of a unique invariant probability measure and that the semi-group of the process converges at exponential rate to this measure. The irreducibility of the process is non trivial because the randomness only derives from the jumps of the slow component. Our proof relies on a fine analysis of the trajectories of the process. 

Further on, we rescale the predator dynamics by dividing the coefficients of the predator differential equation by a small parameter $\varepsilon$. This scaling derives from the metabolic theory and illustrates the fact that the predator mass goes to $0$ while the prey mass remains constant. The metabolic theory links the mass of individuals with their metabolic rates. Numerous experimental studies display relationships between the individual mass and the birth and death rates or the community carrying capacity (see Brown and al. \cite{brown2004toward}, Damuth \cite{Damuth81}). Here, we simplify these relationships by assuming that the predator metabolic rates increase as the invert of their mass. 
This slow-fast system converges as $\varepsilon$ goes to $0$ to an averaged process. In the averaged community, the predator population will always be at an equilibrium that depends on the prey number. Therefore the prey population evolves as a birth and death process where the predator impact is constant between jumps.
In this case, computations concerning the stationary behavior of the averaged process are easier because the community is fully described by the discrete dynamics of the prey population.

This paper is organized as follows. In section \ref{sec:model} we present the piecewise deterministic process and the main results of this article. We give the first properties of the piecewise deterministic prey-predator community process and explain how it derives from an individual based prey-predator community process. In section \ref{sec:ergodicity} we study the ergodic properties of {the prey-predator community process}. These properties derive from a Foster-Lyapunov inequality and the irreducibility of the continuous time process and of discrete time samples. In section \ref{sec:averaged} we rescale the dynamics of predators and prove the convergence of the slow-fast prey-predator community to the averaged process. We prove that this averaged community admits an invariant distribution and study the convergence of the sequence of invariant measures of the slow-fast process as $\varepsilon\to0$ with numerical simulations.
Finally we discuss in section \ref{sec:discuss} our results in view of biological and ecological applications.
\section{Model and main results}
\label{sec:model}
\subsection{The piecewise deterministic model}
\label{subsec:model}
We consider a community of prey individuals and predators in which the predator dynamics is faster than the prey dynamics. The community is described at any time by a vector $Z_t=(N_t,H_t)$ where $N_t\in \N$ is the number of living prey individuals at time $t$ and $H_t\in\R_+$ is the density of predators.

We assume that the prey population evolves according to a birth and death process. The individual birth rate is denoted by $b>0$, the individual death rate by $d\ge0$. The logistic competition among the prey population is represented by a parameter $c>0$. The predation intensity exerted at time $t$ on each prey individual is $B H_t$.

The predators density follows a deterministic differential equation whose parameters depend on the prey population. The individual birth rate at time $t$ is $rB N_t$. It is proportional to the amount of prey consumed by the predator. The parameter $r\in(0,1)$ represents the conversion efficiency of prey biomass into predator biomass. The predator individual death rate $D+CH_t$ includes logistic competition among predators ($D\ge0$, $C>0$).

The community dynamics is given by the differential equation
 \ben
\label{eqdiff}
\frac{d}{dt}H_t=H_t(rB N_t -D-C H_t),
\een
coupled with the jump mechanism
\bean
\label{sauts}
\Pro(N_{t+s}=j| &N_t=n, H_t) = bn s+o(s) \quad\quad \text{if } j=n+1, n\ge1,\\
&= n(d+cn+BH_t)s+ o(s) \quad\quad \text{if } j=n-1, n\ge 2,\\
&= 1-(b+d+cn+BH_t)ns +o(s) \quad\quad\text{if } n=j , n\ge2,\\
&= 1-bs +o(s) \quad\quad\text{if } n=j=1,\\
&=0\quad\quad \text{otherwise}.
\eean
Between the jumps of the prey population process $N$, the dynamics of the predator density $H$ is deterministic. If the process $Z_t$ is at a point $(n,h)\in \N^*\times\R_+$ after a jump, then the predator dynamics is governed by the flow $\phi_n$ associated with equation \eqref{eqdiff}. More precisely, $\phi_n$ satisfies:
\ben
\label{eqdiff-phin}
\frac{d}{dt}\phi_n(h,t)=\phi_n(h,t)(rBn-D-C\phi_n(h,t)),\quad \phi_n(h,0)=h.
\een
Then for all $t\ge0$,
\ben
\label{phi_n}
\phi_n(h,t)=\frac{h e^{t(rB n-D )}}{1+\frac{hC }{rB n-D }(e^{t(rB n-D)}-1)} .
\een
For $h>0$, the solution $\phi_n(h,t)$ remains positive for all $t\ge0$ and converges as $t\to\infty$ toward an equilibrium $ h^*_n$ given by
\ben
\label{equilibre}
h^*_n=\frac{rBn-D}{C}\vee 0,
\een
where $a\vee b $ stands for the maximum of $a$ and $b$.
For sake of simplicity we introduce the global flow on $\N^*\times\R_+$  $\phi((n,h),t)=(n,\phi_n(h,t))$ for $(n,h)\in\N^*\times\R_+$ and $t\ge0$.
\\

In the following, the state space of the prey-predator process is denoted by
$E=\N^*\times\R_+$, we also define the subset $E'=\N^*\times(h^*_1,+\infty)$.
A generic point $z\in E$ is a vector $(n,h)$ with $n\in\N^*$ and $h\in \R_+$.
The process $Z_t=(N_t,H_t)_{t\ge0}$ belongs to the class of \textit{Piecewise Deterministic Markov Processes} introduced by Davis (see \cite{davis1993markov}). It is a $E$-valued Markov process whose infinitesimal generator 
\bean
\label{gene}
\mathcal{A}f(n,h)&=h(rBn-D-Ch)\partial_2 f(n,h)+\Bigl( f(n+1,h)-f(n,h)\Bigr) bn\\
& + \Bigl(f(n-1,h)-f(n,h) \Bigr) n(d+cn+Bh)\un_{n\ge2}.
\eean
is well defined for functions $f :\N^*\times\R_+ \to \R$ bounded measurable, continuously differentiable with respect to their second variable with bounded derivative.
The domain of the extended generator \eqref{gene} has been characterized by Davis (Theorem 26.14 in \cite{davis1993markov}).\\
We denote by $P_t$ the transition semi-group and by $\Pro_{(n,h)}$ (or $\Pro_z$) the law of the process with initial condition $z=(n,h)\in E$. 
\\

\begin{rem}
\label{rem_migr}
In this model, we assume that the prey population cannot become extinct since the death rate is $0$ when there is only one prey individual left. When this assumption is not satisfied, the prey population process can be dominated by a population process without predator which evolves as a logistic birth and death process. It is thus absorbed in $0$ in finite time.
The non-extinction assumption for the prey population is biologically relevant for trees-insects communities for example where the tree population rarely disappears thanks to a migration of trees (e.g. seeds driven by the wind). Here we chose to replace the migration probability of new prey individuals by the non extinction of the prey population. This choice allows to describe the prey-predator dynamics only with individual metabolic parameters such as birth and death rates. However, it is possible to include explicitly a migration at rate $m>0$ in the prey population. We therefore define an alternative process $(N^{(m)},H^{(m)})$ whose infinitesimal generator is given by
\bean
\label{gene_migr}
\mathcal{A}^{(m)}f(n,h)&=h(rBn-D-Ch)\partial_2 f(n,h)+\Bigl( f(n+1,h)-f(n,h)\Bigr) (bn+m)\\
& + \Bigl(f(n-1,h)-f(n,h) \Bigr) n(d+cn+Bh).
\eean
The process $(N^{(m)},H^{(m)})$ is a piecewise deterministic Markov process taking its value in $E^{(m)}=E\cup\{0\}\times \R_+$.
In the sequel, we will mention when our results hold for this alternative model including migration and explain the modification induced by the migration.\end{rem}

\subsection{Main results}
\label{subsec:main_results}
Our main questions on the prey-predator process are twofold. First we are interested in the long time behavior of the process $Z$. 
\begin{theorem}
\label{thm:exp-ergodicity_intro}
The community process $(Z_t)_{t\ge0}$ is exponentially ergodic. It converges toward its unique invariant probability measure $\pi$ at an exponential rate.
There exist $0<\rho<1$ and $0<R<\infty$ such that, for all $ z=(n,h)\in E'$,
\be
\label{expergodic_intro}
\norm{P_t(z,.)-\pi}_{TV} \le R\rho^t (n+hr^{-1}),
\ee
\end{theorem}
Section \ref{sec:ergodicity} is devoted to the proof of the exponential ergodicity of $Z$. Our theorem relies on the theory of Harris recurrent processes, whose main results are recalled in Section \ref{subsec:rappel}. 
In the setting of piecewise deterministic processes these results have been used to derive ergodic properties of different processes (e.g. an additive
increase multiplicative decrease process  \cite{grigorescu2014critical}, a stress release process \cite{last2004ergodicity}, or a wealth-employment process
\cite{bayer2011existence}). There exists also general results proven by \cite{benaim2015qualitative} in the specific case where the deterministic dynamics admits a compact positive invariant set. The case of the prey-predator process $Z=(N,H)$ is more complex since the process is in dimension 2 and neither the jump rates nor the deterministic trajectories are bounded.
\medskip

Our second main result concerns a scaling limit of the process $Z$ corresponding to the biological assumption that the predator mass is small. We introduce a sequence of processes $Z^{\varepsilon}$ for the rescaled parameters 
$r^{\varepsilon}=r/\varepsilon$, $D^{\varepsilon}=D/\varepsilon$ and $C^{\varepsilon}=C/\varepsilon$. As $\varepsilon$ tends to $0$, the dynamics of the fast component $\he$ is accelerated between the jumps of the slow component $\ne$. Therefore we expect that in the slow-fast limit, the prey population only depends on the equilibrium $h^*_n$ of the predator density.
The following theorem is a simplified version of our convergence result stated in Theorem \ref{thm:cv-moy}
\begin{theorem}
 \label{thm:cv-moy_intro}
Fix $ T>0$ and assume $rB-D>0$. We suppose that the sequence of initial conditions $(\Ze_0)_{0<\varepsilon\le1}$ converges to $\widebar{Z}_0$ in law
and moreover that
\be
\label{hyp:moment2_intro}\sup_{0<\varepsilon\le1}\E((\ne_0)^4)<\infty
, \quad \sup_{0<\varepsilon\le1}\E((\he_0)^4)<\infty.
\ee
Then the sequence $\ne$ converges in law toward $\av$ in $\mathbb{D}([0,T],\N)$  as $\varepsilon\to 0$. \\
The process $\av$ is a pure jump process on $\N^*$ whose infinitesimal generator is well defined for every measurable and bounded function $f:\N^*\to\R$ by
\ben
\label{gen-lim_intro}
\mathcal{L} f(n)=(f(n+1)-f(n))bn+(f(n-1)-f(n)) n(d+cn+Bh^*_n)\un_{n\ge2}.\een
\end{theorem}
The proof of the convergence, stated in Section \ref{sec:averaged} relies on a compactness-identification technique. The result proven in Section \ref{sec:averaged}, Theorem \ref{thm:cv-moy}, also states the convergence in law of the sequence of occupation measures associated with the fast component $\he$ (see \cite{kurtz1992averaging}). The main interest of this result is that in the limit, the study of the prey-predator community is simplified, since it is entirely described by the one dimensional birth and death process $\av$. 
\subsection{Pathwise construction and first properties}
\label{subsec:1stproperties}

\noindent Following Fournier and M\'el\'eard (\cite{FM04}) we construct a trajectory of the prey-predator process $Z$ as a solution of a stochastic differential equation driven by Poisson point processes. Let $Q_1(ds,du)$ and $Q_2(ds,du)$ be two independent Poisson point measures on $\R_+\times\R_+$ with intensity $dsdu$ the product of Lebesgue measures.
We define for any initial condition $(n,h)\in E$ the coupled dynamics
\ben
\label{eq:poisson}
 \begin{aligned}
&N_t=n+\int_0^t\int_{\R_+} \un_{u \le b N_{s-}} Q_1(ds,du)\\
&\hspace{2cm}-\int_0^t\int_{\R_+} \un_{u \le N_{s-}(d+cN_{s-}+BH_{s-})\un_{N_{s-}\ge2}} Q_2(ds,du),\\
&\frac{d}{dt}H_t=H_t\bigl( rBN_t-D-CH_t\bigr),\quad \quad H_0=h.
\end{aligned}\een
A unique solution of these equations exists as long as the number of individuals remains finite.
\begin{theorem}
\label{thm:1}
Under the assumption that there exists $p\ge1$ such that 
\be
\label{hyp:moment}
\E\Bigl((N_0)^p\Bigr)<\infty \quad \text{and}\quad \E\Bigl((H_0)^p\Bigr)<\infty,
\ee
\begin{enumerate}[i)]
\item For all $T>0$
\be
\label{csq:moment}
\E\Bigl(\sup_{t\in[0,T]}(N_t)^p\Bigr)<\infty \quad \text{and}\quad \E\Bigl(\sup_{t\in[0,T]}(H_t)^p\Bigr)<\infty,
\ee
\item  If $p\ge1$, there exists a unique solution $(Z_t)_{t\ge0}\in\mathbb{D}(\R_+, E)$ of \eqref{eq:poisson}. Its infinitesimal generator is given by \eqref{gene} for any bounded measurable functions $f$ with $f(n,.)\in\mathcal{C}^1_b(\R_+)$ for all $n\in\N^*$.\\
Moreover, the process $Z_t$ is a Feller process in the sense that for any $g:E\to\R$ continuous and bounded, the function $z\mapsto \E_z(g(Z_t))$ is continuous and bounded on $E$, $\forall t\ge0$.
\item If $p\ge2$, for all bounded measurable functions $f$ with $f(n,.)\in\mathcal{C}^1_b(\R_+)$ for all $n\in\N^*$ and for all $z\in\N^*\times\R_+$,
\be
M_t^f=f(Z_t)-f(z)-\int_0^t \Acal f(Z_s)ds,
\ee
is a $L^2-$martingale starting at $0$ with quadratic variation
\bea
\langle M^f\rangle_t =\int_0^t &\Bigl(f(N_s+1,H_s)-f(N_s,H_s)\Bigr)^2 bN_s \\&+\Bigl(f(N_s-1,H_s)-f(N_s,H_s)\Bigr)^2N_s(d+cN_s+BH_s)  ds.
\eea
\end{enumerate}
\end{theorem}
\begin{proof}
$(i)$ Let us remark that the process $(N_t, t\ge0)$ is stochastically dominated by a pure birth process ($\widetilde{N}_t,t\ge0)$ that jumps from $n$ to $n+1$ at rate $bn$. 
From Theorem 3.1 in \cite{FM04} we know that for all $T>0$
$\E\bigl(\sup_{t\in[0,T]}(\widetilde{N}_t)^p\bigr)<\infty$ and thus $\E\bigl(\sup_{t\in[0,T]}(N_t)^p\bigr)<\infty$.\\
Concerning the predator density, we notice that for all $(n, h)\in \N^*\times\R_+$ the solution $\phi_n(h,t)$ of \eqref{eqdiff-phin} satisfies 
\ben
\label{majoration-phi_n}
\forall t\ge 0, \quad 0<\phi_n(h,t)\le h\vee h^*_n,
\een
Since $h^*_n\le rBn/C$, we obtain that for all $t\ge0$
$$H_t\le H_0\vee \frac{rB}{C}\sup_{s\in[0,t]} N_s.$$
Then
$$ \E\Bigl(\sup_{t\in[0,T]}(H_t)^p\Bigr)\le \E\bigl((H_0)^p\bigr) +\frac{rB}{C}\E\Bigl(\sup_{t\in[0,T]}(N_t)^p\Bigr)<\infty.$$

The fact that the infinitesimal generator is given by \eqref{gene} and the proof of $(iii)$ can be easily adapted from \cite{FM04}.

It remains to prove that $Z_t$ is a Feller process. We adapt the method introduced by Davis \cite{davis1993markov}. The prey-predator community process differs from Davis' setting since the jump rates of the prey population are not bounded. However, we overcome this difficulty using the moment properties given in $(i)$.
We denote by $(T_1, T_2, \dots)$ the sequence of jump times of the prey population. It is always well defined since the jump rate admits a positive lower bound $b$.
Let $g\in\mathcal{C}_b(E)$ and $\psi \in\mathcal{C}_b(E\times\R_+)$. We define the application $G_{\psi}$ on $E\times\R_+$ by
\bea
G_{\psi}(z,t)&=\E_z\Bigl(g(Z_t)\un_{t\le T_1} +\psi(Z_{T_1},t-T_1)\un_{t>T_1} \Bigr). 
\eea
Let $e_1=(1,0)$ be the first vector of the canonical basis on $E$ and let us define a function $\Theta$ on $E\times\R_+$ by $\Theta(z,t)=\int_0^t \theta(\phi(z,s))ds $ with 
\ben
\label{theta}
\theta(z)=\theta(n,h)=n(b+d+cn+Bh).
\een 
The function $t\mapsto 1-e^{-\Theta(z,t)}$ is the cumulative distribution function of the first jump time $T_1$ conditionally on $\{Z_0=z\}$. Then
\bea
G_{\psi}(z,t)&=e^{-\Theta(z,t)}g(\phi(z,t))+\int_0^t \int_E e^{-\Theta(z,t)} \Bigl[  \psi(\phi(z,s)+e_1,t-s) bn+\\
&\quad\quad\quad \psi(\phi(z,s)-e_1,t-s) n(d+cn+B\phi_n(h,s)) \Bigr]ds . 
\eea
Let us remark that $z\mapsto G_{\psi}(z,t)$ is continuous since $z\mapsto \phi(z,t)$ is continuous by Cauchy Lipschitz theorem for all $t\ge0$, and the integrand is locally bounded.\\
We now iterate the kernel $G_{\psi}$. From Lemma (27.3) in \cite{davis1993markov} we get that $\forall k\in \N$, 
$$
{G_{\psi}^{k+1}(z,t)= G^k_{G_\psi}(z,t) }=\E_z\Bigl( g(Z_t)\un_{t\le T_{{k+1}}} +\psi(Z_{T_{{k+1}}},t-T_{{k+1}})\un_{t>T_{{k+1}}} \Bigr).
$$
Then
\bea
\label{cv-P_t}
|G^k_{\psi}(z,t)-\E_z(g(Z_t))| &\le \E_z\Bigl(\Bigl(|g(Z_t)| +|\psi( Z_{T_k},t-T_k) |\Bigr)\un_{t>T_k} \Bigr) \\
&\le (\norm{g}_{\infty} +\norm{\psi}_{\infty}) \Pro_z(T_k<t).
\eea
We deduce from \textit{(i)} with $p=1$, that the sequence of jump times $(T_k)_{k\in \N}$ converges almost surely to $\infty$, hence $\Pro_z(T_k<t)\longrightarrow_{k\to \infty}0$.
To obtain the continuity of $z\mapsto P_tg(z)$ it is sufficient to  prove that the probability $\Pro_z(T_k<t)$ converges to 0 uniformly on compact sets of $E$.\\
Let $K$ be a compact set of $E$, and set $N_{+}=\sup\{n; (n,h)\in K\}$ and $H_{+}=\sup\{h;(n,h)\in K\}$.
We construct a sequence of jump times $(S_k)_{k\in\N}$ that stochastically dominates the sequence of jump times $(T_k)_{k\in\N}$ for any initial condition in $K$.
We start by bounding from above the prey and the predator populations. Similarly as above, we define a prey pure jump process $(X_t)_{t\ge0}$ starting from $N_+$ and a deterministic predator population process $Y_t$ starting from $H_+$:
\bea
&X_t=N_{+}+\int_0^t\int_{\R_+} \un_{u \le b X_{s-}} Q_1(ds,du)\\
&\frac{d}{dt}Y_t=Y_t (r B X_t-D -C Y_t ),\quad Y_0=H_{+}.
\eea
The difference with point $(i)$ lies in the fact that this coupling bounds from above every trajectory with initial condition in $K$: more precisely for any initial condition $z\in K$, $N_t\le X_t$ and $H_t\le Y_t$ for all $t\ge0$, almost surely.
\\
We introduce a Poisson point process with intensity 
$\theta(X_t,Y_t)dt$ and denote by $(S_i)_{i\in \N}$ its sequence of jump times.
Since the rate function $\theta$ increases, we deduce that for all $z\in E$ and $t>0$,
\be
\Pro_z(T_k<t)\le \Pro(S_k<t).
\ee
The probability $\Pro(S_k<t)$ converges toward $0$ as $k\to \infty$, since for all $T>0$, $\E(\sup_{s\in [0,T]}X_s)<\infty$ and $\E(\sup_{s\in [0,T]}Y_s)<\infty$.
\end{proof}

\subsection{Derivation from an individual-based model}
\label{subsec:IBM}
In this part, we justify that the model \eqref{eqdiff}-\eqref{sauts} derives from a microscopic model for the prey-predator community.
We introduce a scaling parameter $K$ tending to $\infty$ and consider that the number of predator is of order $K$ while the prey number remains of order $1$. At each time $t\ge0$, the microscopic community is represented by a vector $(N^K_t,H^K_t)$ where $N^K_t\in \N$ is the prey number  and $H^K_t\in \N$ is the number of predators. This process is a two-types continuous time Markov process whose transition rates are given for all $(n,h)\in\N^2$ by
\be
\label{sauts-micro}
\begin{tabular}{cll}
$(n,h) \to$& $(n+1,h)$& \text{at rate} $nb$\\
& $(n-1,h)$& \text{at rate } $n(d+cn+ B^K h)\un_{n\ge2}$ \\
& $(n,h+1)$& \text{at rate } $hr^KB^K n$\\
& $(n,h-1)$& \text{at rate } $h(D^K+C^K h)$.\\
\end{tabular}
\ee
The parameters $B^K$, $r^K$, $D^K$ and $C^K$ are chosen as follows:
\be
B^K=\frac{B}{K},\quad r^K=Kr, \quad D^K=D ,\quad C^K=\frac{C}{K}\cdot
\ee
The predation and the competition among predators are normalized following \cite{FM04,CFM08}. The parameter of conversion efficiency $r^K$ is scaled in order to maintain constant the benefit from predation.\\
We consider the limit as $K\to\infty$ of the rescaled process $(N^K,\frac{H^K}{K})$. 
\begin{theorem}
Assuming that the sequence of initial conditions \linebreak[4]$(N^K_0,\frac{H^K_0}{K})_{K\ge0}$ satisfies 
\ben
\label{hyp:moment-unif}
\sup_{K}\E\Bigl(\Bigl(N^K_0+\frac{H^K_0}{K}\Bigr)^3\Bigr)<\infty .
\een
and converges in law toward $(n_0,h_0)\in \E$, then for all $T>0$ the process $(N^K,\frac{H^K}{K})_{K\ge0}$ converges in law in $\mathbb{D}([0,T], E)$ toward the piecewise deterministic process $(N_t,H_t)$ defined by \eqref{eqdiff}-\eqref{sauts} with initial condition $(n_0,h_0)$.
\end{theorem}
The proof of this theorem is based on a compactness-uniqueness argument which derives from Theorem 3.1 in \cite{crudu2012convergence} and will not be developed here. The moment assumptions \eqref{hyp:moment-unif} ensure that the processes $Z^K$ and $Z$ are well defined and that the assumptions of Theorem 3.1 in \cite{crudu2012convergence} are satisfied. In the latter, the authors prove a similar result for a gene regulatory network in which the chemical reactions occur at slow or fast speed. 
\section{Ergodic properties}
\label{sec:ergodicity}
In this section, we study the ergodic properties of the prey-predator community process $Z$.
We will prove the irreducibility of the process and of specific sampled chains. From these properties and a Foster-Lyapunov criterion, we will show that there exists a unique invariant probability measure and that the process is exponentially ergodic.

\subsection{Some definitions and known results}
\label{subsec:rappel}
Let us first recall some definitions. Let $(X_t)_{t\ge0}$ be a Feller process taking values in $E$ a locally compact and separable metric space. We denote by $\mathcal{L}$ its infinitesimal generator and by $P_t$ its semi-group.
For every $A\in\mathcal{B}(E)$ we set $ \tau_A=\inf\{t\ge0, X_t\in A\}$.\smallskip\\
The process $X_t$  is \textit{irreducible} if there exists a $\sigma-$finite measure $\nu$ on $E$, 
called irreducibility measure, such that for all $A\in\mathcal{B}(E)$
\be
\label{def-irr}
\nu(A)>0 \quad \Longrightarrow\quad  \forall x\in E,\quad \E_x\Bigl(\int_0^{\infty} \un_A(X_t) dt\Bigr) >0.
\ee
The process $X_t$ is \textit{Harris recurrent} if there exists a $\sigma-$finite measure $\mu$ on $E$ such that $\forall A\in\mathcal{B}(E)$ 
$$
\mu(A)>0 \Longrightarrow \quad \forall x\in E,\quad\Pro_x(\tau_A<\infty)=1.
$$ 
A Harris recurrent Markov process is always irreducible (see \cite{meyntweedie2}).\\
Moreover, a Harris recurrent process $X_t$ has an invariant measure $\pi$ (see \cite{azema1967mesure}). In the case where this measure is finite, we say that $X_t$ is \textit{positive Harris recurrent}.\smallskip\\
For continuous time processes, the positive Harris recurrence can be derived from a Foster-Lyapunov inequality satisfied by the infinitesimal generator on some \textit{petite} set. Recall that a set $C\subset E$ is \textit{petite} if there exist a probability measure $\alpha$ on $\R_+$, and a non degenerate measure $\nu_{\alpha}$ on $E$ such that for any $z\in C$
\be
\label{petiteset}
\int_0^{\infty} P_t(z,.)\alpha(dt)\ge \nu_{\alpha}(.).
\ee
For an irreducible Feller process whose irreducibility measure has a support with non empty interior, all compact sets of $E$ are \textit{petite} sets (from Theorem 5.1 and 7.1 in \cite{tweedie1994topological}).\\
We recall sufficient conditions for the positive Harris recurrence of a Feller process.
\begin{thmx}(Theorem 4.2 in \cite{meyntweedie3})
\label{thm:histo1}
Let $(X_t)_{t\ge0}$ be a Feller process taking values in $E$.
If the following conditions are satisfied:
\begin{enumerate}[i)]
\item $X$ is irreducible with respect to some measure whose support has non empty interior.
\item Foster-Lyapunov inequality:  there exist a function $V:E\to[1,\infty[$ such that $\lim_{|z|\to\infty} V(z)=\infty$, a compact set $K\subset E$ and two constants $\delta,\gamma>0$ such that
\be
\mathcal{L}V(z)\le -\gamma V(z)+\delta\un_K(z),\quad \quad \forall z\in E.
\ee
\end{enumerate}
Then $X$ is positive Harris recurrent and there exists a unique invariant probability measure $\pi$. Moreover $\pi(V)<\infty$.
\end{thmx}

\noindent The process $X_t$ is \textit{ergodic} if it has a unique invariant probability measure $\pi$ and if
\be
\label{ergodic}
\lim_{t\to\infty} \norm{P_t(x,.)-\pi}_{TV} =0,\quad \forall x\in E.
\ee
Moreover, $X_t$ is \textit{exponentially ergodic} if there exist a function 
$R:E\to(0,\infty)$ and $0<\rho<1$ such that
\be
\norm{P_t(x,.)-\pi}_{TV} \le R(x)\rho^t ,\quad \forall x\in E.
\ee
In the case of continuous time Markov processes on continuous state spaces, the ergodicity is related to the behavior of skeletons of the process. A skeleton corresponds to a sampling of the continuous time process at some fixed time. For all $\Delta>0$, the $\Delta-$skeleton of $X$ is the Markov chain $(X_{k\Delta})_{k\in\N}$ with transition kernel $P_{\Delta}$.
We recall sufficient conditions for exponential ergodicity of a Feller process. 
\begin{thmx} (Theorem 6.1 in \cite{meyntweedie2} and Theorem 6.1 in \cite{meyntweedie3} )
\label{thm:histo2}
Let $(X_t)_{t\ge0}$ be a Feller process taking values in $E$ which satisfies both conditions i) and ii) in Theorem \ref{thm:histo1}. If furthermore there exists an irreducible skeleton $(X_{k\Delta})_{k\in\N}$ ($\Delta>0$), then $X$ is exponentially ergodic and there exist $0<\rho<1$ and $0<R<\infty$ such that, for all $ z\in E$,
\be
\norm{P_t(z,.)-\pi}_{TV} \le R\rho^t V(z).
\ee
\end{thmx}
Let us briefly explain the origin of the condition on the skeleton of the process. Since $X_t$ is positive Harris recurrent, the irreducible skeleton $(X_{k\Delta})_{k\in\N}$ has an invariant probability measure. Hence, the skeleton chain is positive recurrent and aperiodic (see Theorem 5.1 in \cite{meyntweedie2}). The irreducibility is crucial to obtain the aperiodicity.\\
Moreover, from the Foster-Lyapunov inequality ii) in Theorem \ref{thm:histo1}, we deduce that the skeleton chain also satisfies a Foster-Lyapunov inequality with the same function $V$: there exist $\gamma'<1$ and $\delta'>0$ such that 
for every initial condition $z\in E$
\be
\E_z\bigl( V(X_{\Delta})\bigr)  \le \gamma'V(z)  +  \delta'.
\ee
From Theorem 6.3 in \cite{meyntweedie1}, we deduce that the skeleton $(X_{k\Delta})_{k\in\N}$ is geometrically ergodic. There exist $0<\rho<1$ and $0<R<\infty$ such that for all $z\in E$ 
\be
\norm{P_{k\Delta}(z,.)-\pi}_{TV} \le R\rho^kV(z).
\ee
The exponential ergodicity of the continuous time process then derives from the semi-group property.

\subsection{Irreducibly}
In this section we study the irreducibility of $(Z_t)_{t\ge0}$ in $E=\N^*\times\R_+$. 
Let us highlight that a Borel set $A\in\mathcal{B}(E)$ can always be written as 
$$
A=\bigcup_{k\ge1} \{k\}\times A_k,
$$ 
where $A_k\in\mathcal{B}(\R_+)$.
We introduce the measure $\sigma$ on $E$ as the product of the counting measure on $\N^*$ and the Lebesgue measure $\lambda$ on $\R_+$: 
\ben
\label{sigma}
\forall A\in\mathcal{B}(E),\quad \sigma(A)= \sum_{k\ge1} \lambda(A_k).
\een
In particular, if $\sigma(A)>0$, then there exist $k\in\N^*$ such that $\lambda(A_k)>0$.
\medskip

\begin{theorem}
\label{thm:phiirr}
\begin{enumerate}[(i)]
\item If $rB -D \le 0$, then the process $(Z_t,t\ge0)$ is irreducible for the measure $\sigma$ on $E$ given by \eqref{sigma}.
\item Otherwise, the process $(Z_t,t\ge0)$ is irreducible for the measure $\sigma'$ which is the restriction of $\sigma$ to the space $E'=\N^*\times ( h^*_1,+\infty)$, for $h^*_1$ defined in \eqref{equilibre}.
\end{enumerate}
\end{theorem}

The dichotomy in this result derives from the fact that when $h^*_1>0$, the set $\N^*\times (0,h^*_1)$ is transient for the dynamics. Indeed, the flows $t\mapsto\phi_n(h,t)$ are increasing functions for $(n,h)\in\N^*\times (0,h^*_1)$ and therefore the trajectories cannot enter this area. 

In the sequel we prove a stronger result on the probability for the process $Z_t$ to reach open Borel sets, from which Theorem \ref{thm:phiirr} follows.
\begin{theorem}
\label{thm:irr+}\begin{enumerate}[(i)]
\item In the case where $rB -D \le 0$, we consider an interval $I=\{k\}\times(h_-,h_+)$ with $0\le h_-<h_+$ and $k\in\N^*$. Then for every initial condition $(n,h)\in E$, there exists $t_0>0$ such that $\forall t\ge t_0$,
\be
\Pro_{(n,h)}\bigl(Z_t\in I)>0.
\ee
\item
We have a similar result in the case where $rB -D >0$ for any interval $I\subset E'$ such that $\sigma'(I)>0$ and any initial condition $(n,h)\in E'$.
\end{enumerate}
\end{theorem}
 
The proof derives from the construction of ideal trajectories and from comparisons between the different predator flows.

\begin{proof}[Proof of Theorem \ref{thm:irr+}]
\textit{(i)} We assume that $rB-D\le0$ which is equivalent to $h^*_1=0$ .\medskip\\
We consider different cases depending on the position of the interval $I=\{k\}\times(h_-,h_+)$ with respect to the line $n\mapsto h^*_n$ of the predator equilibria and on the initial condition $(n,h)\in E$. These cases are illustrated on Figures \ref{fig:1} to \ref{fig:3emecas}. On these Figures, the state space $E$ is represented as the positive quadrant of $\R^2$ separated by the line $n\mapsto h^*_n$. The process $(Z_t,t\ge0)$ can only cross this line by a jump of the prey number. When the process is above this line, the predator density decreases, while it increases when the process is under this line.
The intervals which cross this line (i.e. such that $h^*_k\in(h_-,h_+)$) will play a specific role in the proof since they are stable by the predator flow $\phi_k$. \\
We introduce additional notations:
for all $m\in\N^*$, and $x,y\in\R_+$ we set $r^m(x,y)$ the time needed for the flow $\phi_m$ to go from $x$ to $y$. This time is well defined for $x\ge y > h^*_m$ or $x\le y< h^*_m$. In these cases, it satisfies
\bea
&\phi_m(x,r^m(x,y))=y.
\eea

\noindent\textbf{First case: the interval $I$ is stable for the flow $\phi_k$ (i.e. $h^*_k\in(h_-,h_+)$ if $h^*_k>0$ or $h_-=h^*_k=0$ otherwise)} 
\\
Our aim is to prove that for any $t>t_0$, $\Pro_{(n,h)}(Z_t\in I)>0$ for some $t_0\ge0$. The idea is to construct simple trajectories which enter the interval $I$ and arise with positive probability. \\
We split the reasoning into different sub-cases depending on the initial condition. We focus on initial conditions such that $h>h^*_n$. The other cases can be treated similarly by symmetry. 
\smallskip\\
\underline{$A$) If $n\le k$ and $h_-\vee h^*_n<h$.}\smallskip\\
\indent We first consider the specific sub-case where $n\le k$ and $h_-<h^*_n\le h\le h_+$ (see Figure \ref{fig:1}).\\
In this setting, we are interested in trajectories with exactly $k-n$ prey births. 
These trajectories reach the line $\{k\}\times\R_+$.
Furthermore, the number of predators remains in the interval $[h^*_n,h\vee h^*_k]\subset (h_-,h_+)$. This property derives from the fact that the predator density decreases as long as $H_t\ge h^*_{N_t}$ and remains therefore smaller than $h$ but greater than $h^*_{N_t}\ge h^*_n$ since $N_t\ge n$. If the process jumps below the line $n\mapsto h^*_n$ then the predator density increases and remains bounded by $h^*_k$.
Thus, after $k-n$ births events, the process reaches the interval $I$.
\begin{figure}[h!]
\begin{center}
\scalebox{0.6}{
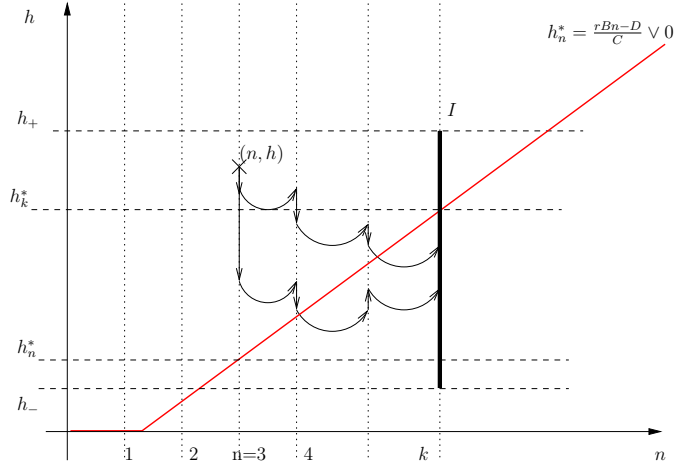}
\end{center}
\caption{\small{(Case $1 A.$) Different ideal trajectories for the specific sub-case where the initial condition $(n,h)$ satisfies $n\le k$ and $h_-<h^*_n\le h\le h_+$. The red line is the map $m\mapsto h^*_m$.}}
\label{fig:1}
\end{figure}

\noindent Let us now prove that such trajectories occur with positive probability.
We first compute the probability that the first jump is a birth :
\bea
\Pro_{(n,h)}(\text{first jump is a birth})&=\E_{(n,h)}(\E(\un_{\text{first jump is a birth}}\lvert T_1))\\
&=\E_{(n,h)}\Bigl(\frac{bn}{\theta(n,\phi_n(h,T_1))}\Bigr)\\
&\ge \frac{bn}{\theta(n,\phi_n(h,h\vee h^*_n))},
\eea where the total jump rate $\theta(n,h)$ defined in \eqref{theta} increases in $n$ and $h$.\\
Then by induction, the probability that the $k-n$ first jumps are births, is greater than
 \be
\frac{bn}{\theta(n,h\vee h^*_n)}\times \frac{b(n+1)}{\theta(n+1,h\vee h^*_n)} \times \cdots \times  \frac{b(k-1)}{\theta(k-1,h\vee h^*_n)}>0.
\ee
Recall that the sequence of jump times of the prey population is denoted by $(T_m)_{m\in\N}$. Let us fix $t>0$.
Using the lower bound $h^*_n$ of the predator population size, we bound from below the probability that the $k-n$ births happen before $t$ by
\be
{\Pro_{(n,h)}}\bigl(T_{k-n}<t \big| k-n\text{ births}\bigr)\ge \Pro\bigl(\mathcal{P}oiss(t\theta(n,h^*_n)\bigr)\ge k-n)>0.
\ee
where $\mathcal{P}oiss(t\theta(n,h^*_n))$ is a random variable with Poisson distribution of parameter $t\theta(n,h^*_n)$.
Finally, we request that no other jump occurs before $t$, then
\be
\Pro{(n,h)}\bigl(T_{k-n+1}>t \big|T_{k-n}<t\text{ and } k-n\text{ births}\bigr)\ge \exp(-t \theta(k,h_+))>0
\ee 
Then the event $\bigl\{T_{k-n+1}>t \text{ and }T_{k-n}<t\text{ and } k-n\text{ births} \bigr\}$ has positive probability, and on this event $Z_t\in I$.
\medskip\\

Let us come back to the general case where $n\le k$ and $h_-\vee h^*_n<h$.\\
We will consider the trajectories which remain on $\{n\}\times\R_+$ until $H_t$ reaches $h_+$. Then, we will request that $k-n$ births occur before the predator population size reaches $h_-$.\\
Therefore, we define $r_1=r^n(h,h_+)$ when it exists and set $r_1=0$ otherwise. The first step is to require that $T_1>r_1$. Since in this case $h_+>h^*_k\ge h^*_n$ and $h> h^*_n$, the flow $\phi_n(h,.)$ decreases and thus
\be
\Pro_{(n,h)}(T_1>r_1)=\exp\Bigl(-\int_0^{r_1}\theta(n,\phi_n(h,s))ds\Bigr)\ge \exp(-r_1\theta(n,h))>0.
\ee 
Then, we define $r_2=r^n(h_+,h_-)$ when it exists (i.e. if $h^*_n\in (h_-,h_+)$) and set $r_2=+\infty$ otherwise. It is important to remark that $r_1+r_2=r^n(h,h_-)$. The specific case considered above corresponds to $r_1=0$ and $r_2=+\infty$.\\
We request that $T_{k-n}<r_1+r_2$ and that these $k-n$ jumps are births. An easy adaptation of the previous result shows that this event has positive probability.\\
Moreover, at time $T_{k-n}$, $H_{T_{k-n}}\in(h_-,h_+)$. The upper bound $H_{T_{k-n}}<h_+$ derives from the same reasoning as above. The lower bound of the predator density comes from comparisons of the different flows. We denote by $F_l$ the vector field associated with $\phi_l$: $F_l(y)=y(rBl-D-Cy)$, for $l\in\N^*$.\\
The first birth occurs at time $r_1<T_1<r_1+r_2$ and $H_{T_1}=\phi_n(h,T_1)>h_-$. The second jump happens at $T_2<r_1+r_2$ and $H_{T_2}=\phi_{n+1}(H_{T_1},T_2-T_1)>h_-$. 
Since $F_{n+1}(y)>F_{n}(y)$ for all $y\in\R_+$ then 
$\phi_{n+1}(H_{T_1},T_2-T_1)\ge \phi_n(H_{T_1},T_2-T_1)$ and thus
$H_{T_2}\ge \phi_n(h,T_2)$.
Then by iteration, we deduce that $H_{T_{k-n}}\ge\phi_n(h,T_{k-n})>\phi_n(h,r_1+r_2)\ge h_-$.\\
We define the time $t_0=r_1+r_2\un_{r_2<+\infty}$. Let us now consider $t>t_0$ and finally request that $T_{k-n+1}>t$. 
As in the previous case, we deduce that these trajectories occur with positive probability and satisfy $Z_t\in I$.
\begin{figure}[h!]
\begin{center}
\scalebox{0.6}{
 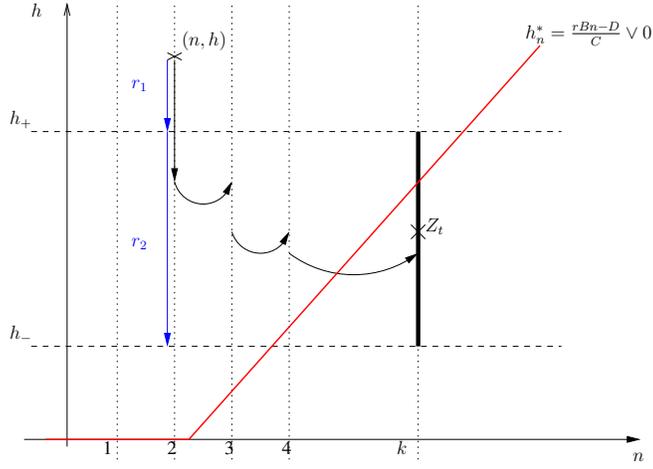}
\end{center}
\caption{\small{(Case $1 A.$) An ideal trajectory. The vertical blue arrows represent the time needed by the flow to get from the tail to the head of the arrow.}}
\label{fig:2}
\end{figure}
\medskip\\
\noindent\underline{$B$) If $n\le k$ and $h^*_n\le h\le h_-$.} (See Figure \ref{fig:3})\smallskip\\
The challenge is to increase the predator density up to $h_-$. 
Let us fix a time $s>0$. We consider trajectories which have exactly $k-n$ jumps before $s$, which are births. Then using a similar reasoning to case $A)$, we deduce that $H_s\ge \phi_n(h,s)$. We define the time $r_1=r^k(\phi_n(h,s),h_-)$. 
Therefore, for every $t>t_0=s+r_1$, if no jump occurs on the time interval $[s,t]$, then
\be
H_t=\phi_k(H_s,t-s) > {\phi_k}(H_s,r_1),
\ee
since $t-s> r_1$. Moreover, $\phi_k(H_s,r_1)\ge \phi_k(\phi_n(h,s),r_1)=h_-$ as $H_s\ge \phi_n(h,s)$.
Thus, $Z_t\in I$.

\begin{figure}[h!]
\begin{center}
\scalebox{0.6}{
 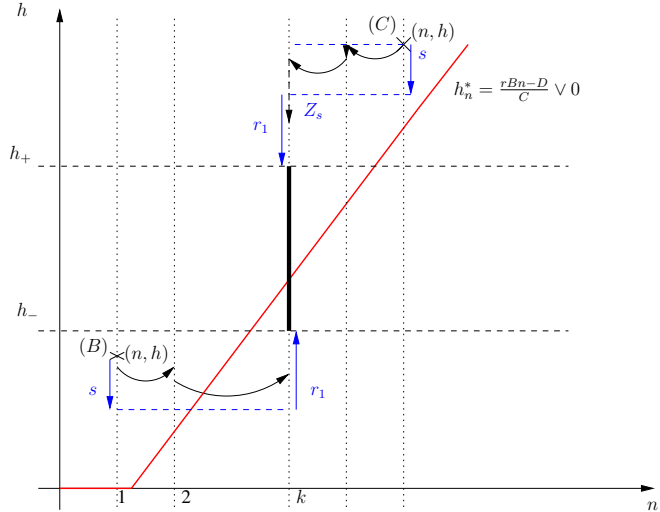}
\end{center}
\caption{\small{(Cases $1B.$ and $1C.$) Examples of ideal trajectories. }}
\label{fig:3}
\end{figure}

\noindent\underline{$C$) If $n>k$ and $h>h^*_n$.} (See Figure \ref{fig:3})\smallskip\\
The reasoning is similar to the previous case, except that we aim at decreasing the predator density.  We consider trajectories which have exactly $n-k$ deaths before $s>0$. Then $H_s\le \phi_n(h,s)$. We define the time $r_1=r^k(\phi_n(h,s),h_+)$ when $\phi_n(h,s)\ge h_+$ and set $r_1=0$ otherwise. 
For every $t>t_0=s+r_1$, if no jump occurs on the time interval $[s,t]$, then $Z_t\in I$.
\bigskip\\

\noindent\textbf{Second case: The interval $I$ is below the line $n\mapsto h^*_n$(i.e. $h^*_k>h_+$) }\\
We will construct an auxiliary interval which is stable for the predator flow. Then, we will prove that starting from this interval, the process enters $I$ in some finite time.\\
We introduce the integer $m=\max\{l\in\N^*,h^*_l<h_+\}$. Once again, we split the reasoning in three cases depending on the position of $h^*_m$ with respect  to $(h_-,h_+)$ and the positivity of $h^*_m$.\smallskip\\
\underline{A) If $h_-< h^*_m<h_+$.}\smallskip\\
We define the interval $J=\{m\}\times(\frac{h_+ +h^*_m}{2},h_-)$ which is stable for the flow $\phi_m$ (see Figure \ref{fig:4}). Then, from the first case, there exists $t_0$, such that $\forall t\ge t_0$, $\Pro(Z_{t}\in J)>0$.\\
We set $r_1=r^k(\frac{h_+ + h^*_m}{2},h_+)$. Let us remark that the trajectories starting from $(n_0,h_0)\in J$ such that exactly $k-m$ births occur during $r_1$ satisfy that $Z_{r_1}\in I$. This derives once again from comparisons of the flows $\phi_l$ for $m\le l\le k$. Moreover, such trajectories arise with positive probability.\\
Therefore for any $t\ge t_0+r_1$, we deduce from the Markov property at time $t-r_1$ that
\be
\Pro_{(n,h)}\bigl(Z_{t}\in I\bigr)\ge\E_{(n,h)}\Bigl(\Pro_{Z_{t-r_1}}(Z_{r_1}\in I)\un_{Z_{t-r_1}\in J}\Bigr) \Pro_{(n,h)}\bigl(Z_{t-r_1}\in J\bigr)>0.
\ee
\begin{figure}[h!]
\begin{center}
\scalebox{0.6}{
 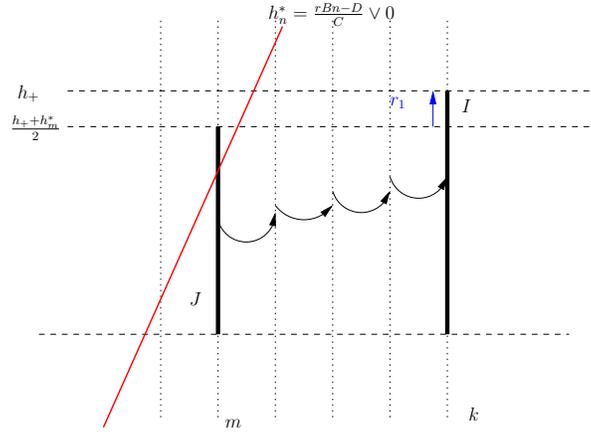}
\end{center}
\caption{\small{(Case $2A.$) Construction of the auxiliary interval and of an ideal trajectory.}}
\label{fig:4}
\end{figure}

\medskip
\noindent\underline{B) If $0<h^*_m\le h_-<h_+$.}\smallskip\\ 
For this configuration we use the invertibility of the flow $\phi_k$. 
We will construct an interval $I'=\{k\}\times (u,v)$ such that
$\phi_k(u,s_0)=h_-$ and $\phi_k(v,s_0)=h_+$ for some $s_0>0$ and that furthermore satisfies that $u<h^*_m<v$.
\begin{figure}[h!]
\begin{center}
\scalebox{0.6}{
 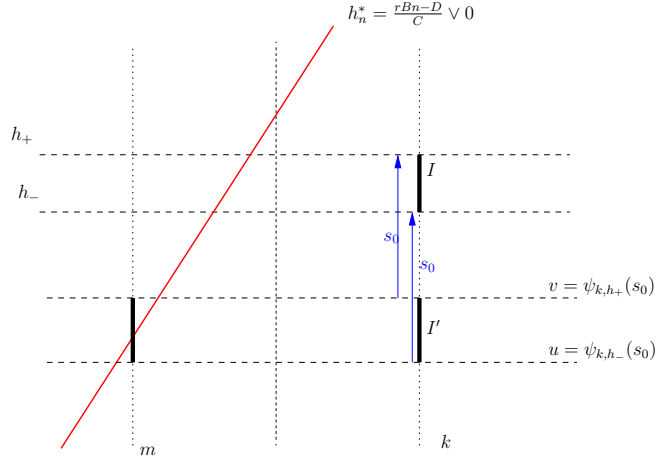}
\end{center}
\caption{\small{(Case $2B.$) Construction of the auxiliary interval $I'$.}}
\label{fig:5}
\end{figure}

\noindent To this aim, we fix $\varepsilon\in(0,h^*_m)$ and remark that for any $k\ge 1$ and $0<\varepsilon<y\le h<h^*_k$, the equation $\phi_k(y,s)=h $ is equivalent to $y=\psi_{k,h}(s)$, where
$\psi_{k,h}(s)$ is the inverse image by the flow $\phi_k(.,s)$ of the point $h$. We deduce from \eqref{phi_n}  that the application $\psi_{k,h}$ is defined from $[0,r^k(\varepsilon,h)]$ to $[\varepsilon,h]$ by
\be
\psi_{k,h}(s)=\frac{h}{e^{s(rBk-D)}-\frac{hC}{rBk-D}(e^{s(rBk-D)}-1)}.
\ee 
It is continuous and strictly decreasing on $[0,r^k(\varepsilon,h)]$.\\
Furthermore, from the uniqueness of the flow we deduce that for any \linebreak[4] $ r^k(\varepsilon,h)\ge s\ge0$
$$\psi_{k,h_+}(s) > \psi_{k,h_-}(s).$$
Therefore, there exists a time $s_0>0$ such that the points $v=\psi_{k,h_+}(s_0)$ and $u=\psi_{k,h_-}(s_0)$ satisfy $u<h^*_m<v$ and we set $I'=\{k\}\times (u,v)$. \\
From the case $2.A)$ we deduce that there exists $t_1$ such that $\forall t\ge t_1$, 
$$
\Pro_{(n,h)}(Z_{t}\in I'\bigr)>0.
$$ 
For any trajectory which is in $I'$ at time $t$, we request that no jump occurs during $s_0$, which happens with positive probability. Therefore, using the Markov property at time $t$, we deduce that
$$\Pro_{(n,h)}(Z_{t+s_0}\in I\bigr)>0.$$
\medskip
\underline{C) If $h^*_m=0$}\smallskip\\
In this case the above construction 2B) does not work because the only stable interval on $\{m\}\times\R_+$ are of the form $\{m\}\times(0,a)$ with $a>0$, which would impose $s_0=+\infty$.\\
Let us fix a small $\delta>0$ and define the interval $I^3=\{m\}\times (h_-+\delta,h_+-\delta)$. We remark that we can adapt the previous reasoning to prove that for every $(n,h)\in E$ there exists $t_0>0$, such that $\forall t\ge t_0$ ,
$$\Pro_{(n,h)}(Z_t\in I^3)>0.$$ 
\begin{figure}[h!]
\begin{center}
\scalebox{0.7}{
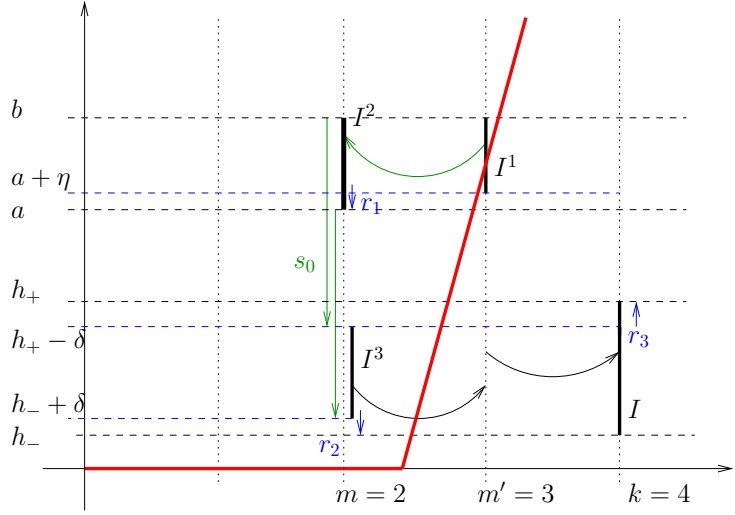
}
\caption{\small{(Case 2C.) Construction of the auxiliary intervals $I_1$, $I_2$ and $I_3$.}}
\label{fig:cas2C}
\end{center}
\end{figure}

Let us explain this construction (see Figure \ref{fig:cas2C}). We first define $m'=\min\{ q\in\N^*, h^*_q > h_-+\delta\}$. As in the step 2B), we construct an auxiliary interval $I^2=\{m\}\times (a,b)$ with $s_0>0$
\bea
&h_+-\delta =\phi_m(b,s_0),\\
&h_-+\delta =\phi_m(a,s_0),\\
&\text{and } a<h^*_{m'} < b.
\eea
We fix $\eta>0$ such that $h^*_{m'}> a+\eta$ and set $I^1=\{m'\}\times (a+\eta, b)$. From the first step there exists $t_1$ such that $\forall t\ge t_1$, $\Pro_{(n,h)}(Z_t\in I^1)>0$. Starting from $I^1$ we request furthermore that exactly $m'-m$ successive deaths occur on the time interval $[t,t+r_1]$ with $r_1=r^m(a+\eta,a)$. This ensures that $Z_{t+r_1}\in I^2$. Furthermore we request that no jump occurs on $[t+r_1,t+r_1+s_0]$, and thus, $\forall t\ge t_1$ $\Pro_{(n,h)}(Z_{t+r_1+s_0}\in I^3)>0$.\\

We now define the times $r_2=r^{m}(h_-,h_-+\delta)$ and $r_3=r^k(h_+-\delta,h_+)$ and set $t_2=\min(r_2,r_3)$.
For any trajectory which is at time $t+r_1+s_0$ in $I^3$ we request furthermore that exactly $k-m$ successive births occur before the time $t+r_1+s_0+t_2$. Therefore, we deduce from the Markov property that $\forall t\ge t_0$
$$\Pro_{(n,h)}(Z_{t+r_1+s_0+t_2}\in I)>0.$$

\noindent\textbf{Third case: $h^*_k< h_- $}.\\
The proof is very similar to the second case. We introduce the smallest integer $m$ such that $h^*_m> h_-$ and adapt the previous reasoning by inverting birth and death events.
\\

\textit{(ii)} Let us now consider the situation where $h^*_1>0$. 
Starting from a point $(n_0,h_0)\in E$ such that $h\ge h^*_1$, the process cannot reach the set $\{z\in E, h\le h^*_1\}$ which corresponds to the hatched zone on Figure \ref{fig:3emecas}. Therefore, we restrict ourselves to the measure $\sigma'$ and initial conditions in $E'$. The proof is the similar to above.
\begin{figure}[h!]
\begin{center}
\scalebox{0.5}{
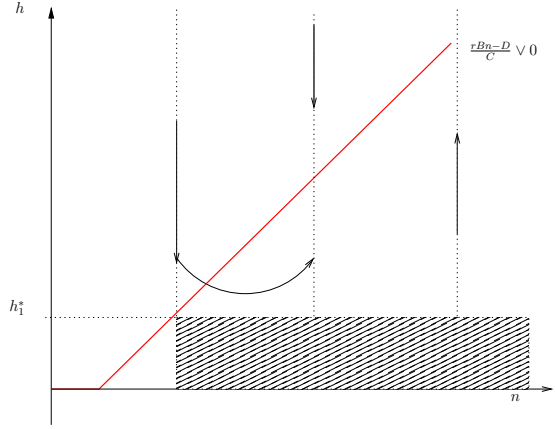
}
\caption{\small{The hatched area corresponds to the points that can not be reached starting from an initial condition which is not in this zone.}}
\label{fig:3emecas}
\end{center}
\end{figure}

\end{proof}

\begin{proof}[Proof of Theorem \ref{thm:phiirr}]
We give the proof in the case where $h^*_1=0$, the other case being an easy adaptation.\\
For any $A\in\mathcal{B}(E)$ such that $\sigma(A)>0$, there exist an integer $k\in\N^*$ and a Borel set $A_k\in\mathcal{B}(\R)$ such that $\{k\}\times A_k\subset A$ and $\lambda(A_k)>0$. Once again, we split the proof in two sub-cases.\\
\textbf{First case: }Let us first assume that there exists an open interval $(h_-,h_+)\subset A_k$ with $h_-<h_+$ and define $I=\{k\}\times(h_-,h_+)$.\\We choose $\varepsilon$ small enough such that the interval $I_{\varepsilon}=\{k\}\times(h_-+\varepsilon,h_+-\varepsilon)$ still satisfies $\sigma(I_{\varepsilon})>0$ and fix a small $\delta>0$ such that $\forall (k,h')\in I_{\varepsilon}$, $\phi_k(h',\delta)\in I$.
\\
From Theorem \ref{thm:irr+} we can construct trajectories that belong to the interval $I_{\varepsilon}$ at time $t\ge t_0$ with positive probability for some $t_0>0$. We ask furthermore that no jump occurs during a time $\delta$. Then for any $t\ge t_0$
\bea
\Pro_{(n,h)}\bigl(Z_s\in I,&\forall s\in[t,t+\delta]\bigr)\\
&\ge \Pro_{(n,h)}\bigl(Z_t\in I_{\varepsilon} \text{ and no jump occurs on }[t,t+\delta]\bigr)\\
&\ge e^{-\delta \theta(k,h_+-\varepsilon)} \Pro_{(n,h)}\bigl(Z_t\in I_{\varepsilon}\bigr)>0.
\eea
Therefore
$$\E_{(n,h)}\Bigl(\int_{0}^{\infty} \un_A(Z_s)ds)\Bigr)\ge \delta \Pro_{(n,h)}\bigl(Z_s\in I,\forall s\in[t,t+\delta]\bigr)>0.
$$
\medskip\\
\textbf{Second case: }We now consider the case where $A_k$ doesn't contain any interval (as an example $\R\setminus \mathbb{Q}$, or a fat Cantor set). We consider an open bounded interval $(h_-,h_+)$ such that $\lambda(A\cap (h_-,h_+))>0$. Such an interval always exists since 
$$\lambda(A_k)=\sum_{M=0}^{\infty} \lambda(A_k\cap (M,M+1))>0.$$
Moreover, it is possible to choose $(h_-,h_+)$ such that $h^*_k\notin [h_-,h_+]$, i.e. this interval is not stable for the flow $\phi_k$. Indeed, the opposite case would imply by successive divisions of the interval, that for any $\varepsilon>0$, $\lambda\bigl(A_k \setminus(h^*_k-\varepsilon ,h^*_k+\varepsilon)\bigr) =0$. Thus, we would have that 
$0<\lambda(A_k)=\lambda\bigl(A_k\cap (h^*_k-\varepsilon ,h^*_k+\varepsilon)\bigr)\le 2\varepsilon$ for any $\varepsilon>0$, which is not possible.\\
We now restrict ourselves to the set $B_k=A_k\cap (h_-,h_+)$ with $\lambda(B_k)>0$ and assume that the flow $\phi_k$ increases on $(h_-,h_+)$ (the other case being an easy adaptation).\\
Let us fix $\varepsilon>0$. In the sequel we consider the trajectories that reach the interval $(h_--\varepsilon,h_-)$ and then, we ask that no jump occurs until these trajectories attain $h_+$. Then, the time spent by those trajectories in $B_k$ will be positive since the flows are continuous.
More precisely, from Theorem \ref{thm:irr+} we deduce that there exists $t_0$ such that $\forall t\ge t_0$, $\Pro_{(n,h)}\bigl(Z_t\in \{k\}\times(h_--\varepsilon,h_-)\bigr)>0$. We define the positive time $r_1=r^k(h_--\varepsilon, h_+)$ needed for the flow $\phi_k$ to go from $h_--\varepsilon$ to $h_+$. 
For all $t\ge t_0$, we consider the event
$$\mathcal{E}_t=\Bigl\{ Z_t\in \{k\}\times(h_--\varepsilon,h_-) \text{ and no jump occurs on }[t,t+r_1] \Bigr\}.$$ Then
$\Pro_{(n,h)}\bigl( \mathcal{E}_t\bigr)>0$
and
$$\E_{(n,h)} \Bigl(\int_t^{t+r_1} \un_{\{k\}\times B_k} (Z_s) ds \mid\mathcal{E}_t\Bigr) =\E_{(n,h)} \Bigl(\int_0^{r_1} \un_{ B_k} (\phi_k(Z_t,s)) ds \mid\mathcal{E}_t\Bigr).
$$
Since the flow $\phi_k$ is invertible, we make the change of variable $u=\phi_k(Z_t,s)$ and obtain
\bea
\E_{(n,h)} \Bigl(\int_t^{t+r_1}& \un_{\{k\}\times B_k} (Z_s) ds\mid\mathcal{E}_t \Bigr) \\&=\E_{(n,h)} \Bigl(\int_{Z_t}^{\phi_k(Z_t,r_1)} \frac{\un_{ B_k} (u)}{u(rBk-D-Cu)} du\mid\mathcal{E}_t \Bigr).
\eea
Since, $Z_t\in(h_--\varepsilon,h_-)$ and $\phi_k(Z_t,r_1)\in (h_+,h^*_k)$, we deduce that for some constant $v>0$, 
\be
\E_{(n,h)} \Bigl(\int_t^{t+r_1} \un_{\{k\}\times B_k} (Z_s) ds \mid\mathcal{E}_t\Bigr) >v\lambda(B_k )>0.
\ee
Thus, $\E_{(n,h)}\bigl(\int_{0}^{\infty} \un_{B_k}(Z_s)ds)\bigr)>0$ which concludes the proof of the irreducibility.
\end{proof}

\subsection{Positive Harris recurrence}
We recall the expression of the infinitesimal generator of the prey predator process given in \eqref{gene}. We prove in the following that it satisfies a Foster-Lyapunov criterion.
\begin{prop}
\label{prop:FL}
Let $V: E\mapsto [1,+\infty[$ be the function 
$V(n,h)=n+h/r$.
Then there exist $\delta,\gamma>0$ and a compact set $K$ such that
\ben
\label{drift}
\Acal V(z)\le -\gamma V(z)+\delta\un_K(z), \quad \quad \forall z\in E
\een 
\end{prop}

\medskip
\noindent We combine Theorem \ref{thm:phiirr} and Proposition \ref{prop:FL}, to deduce from Theorem \ref{thm:histo1} in Section \ref{subsec:rappel} that
\begin{theorem}
\label{thm:HR}
The process $Z_t$ is positive Harris recurrent and thus there exists a unique invariant probability measure $\pi$ on $E'$ which furthermore satisfies $\pi(V)<\infty$.  
\end{theorem}

\begin{proof}[Proof of Proposition \ref{prop:FL}]

For all $(n,h)\in E$
\bea
\label{AV}
\Acal V(n,h) &= \frac{h}{r}(rB n-D-C h) +bn-n(d+cn+B h)\un_{n\ge2}\\
&= 
\left\{\begin{aligned}
&-h\Bigl(h\frac{C}{r}-\frac{D}{r}\Bigr) -n(nc -(b-d)),\quad \text{if } n\ge2\\
&-h\Bigl(h\frac{C}{r}-\frac{D}{r}-B\Bigr) +b,\quad \text{if } n=1\\
\end{aligned} \right.
\eea
For any $\gamma>0$, we obtain easily that there exists $n_0\in\N^*$ and $h_0>0$ such that 
\bea
&-n(nc -(b-d))\le -\gamma n,\quad \forall n\ge n_0\\
\text{and }&-h\Bigl(h\frac{C}{r}-\frac{D}{r}\Bigr) \le -\gamma \frac{h}{r},\quad \forall h\ge h_0.
\eea
Then for all $(n,h)$ with $n\ge n_0$ and $h\ge h_0$, $\Acal V(n,h)\le -\gamma V(n,h)$.  
Moreover since $n_0$ is finite, there exists $h_1(n_0)$ such that $\forall (n,h)$ with $n<n_0$ and $h\ge h_1$, $\Acal V(n,h)\le -\gamma V(n,h)$. Similarly, 
there exists $n_1(h_0)$ such that $\forall (n,h)$ with $h<h_0$ and $n\ge n_1$, $\Acal V(n,h)\le -\gamma V(n,h)$.\\
Therefore, setting $K=\{1,\cdots,\max(n_0,n_1)\}\times [0,\max(h_0,h_1)]$ and \linebreak[4] $\delta=\max\{V(n,h),(n,h)\in K\}$, the function $V$ introduced above clearly satisfies the Foster-Lyapunov criterion  \eqref{drift}.
\end{proof}

\begin{rem}
The function $V$ introduced in Proposition \ref{prop:FL} is not the only Lyapunov function of the system. An easy computation leads to the fact that $W(n,h)=n^2+h$ also satisfies \eqref{drift}. With Theorem \ref{thm:HR} we conclude that $\pi(W)<\infty$ which improves our knowledge of the invariant probability measure $\pi$.\end{rem}
\subsection{Exponential ergodicity}
\label{sebsec:experg}
In this section we investigate the convergence in total variation norm of the transition kernel toward the invariant measure. Let us first recall Theorem \ref{thm:exp-ergodicity_intro}.
\begin{theorem*}
\label{thm:exp-ergodicity}[Theorem \ref{thm:exp-ergodicity_intro}]
The community process $(Z_t)_{t\ge0}$ is exponentially ergodic. It converges toward its invariant probability measure $\pi$ at an exponential rate. There exist $0<\rho<1$ and $0<R<\infty$ such that, for all $ z\in E'$,
\ben
\label{expergodic}
\norm{P_t(z,.)-\pi}_{TV} \le R\rho^t V(z) 
\een
\end{theorem*}
\begin{proof}
From Theorem \ref{thm:histo2} in Section \ref{subsec:rappel}, it remains to prove that a skeleton chain of the prey-predator process is irreducible. This condition is actually equivalent to the ergodicity for positive Harris recurrent processes (Theorem 6.1 \cite{meyntweedie2}).\\
 It derives immediately from Theorem \ref{thm:irr+}, that any skeleton chain $(Z_{k\Delta})_{k\in\N}$ (with $\Delta>0$) reaches any open Borel set with positive probability. Indeed, let $O$ be an open set of $\B(E')$ with $\lambda(\mathcal{O})>0$ then there exists an interval $I\subset\mathcal{O}$ with $\lambda(I)>0$ and thus, for any $z\in E'$ and for $q$ large enough, $\Pro_{z}(Z_{q\Delta}\in\mathcal{O})\ge \Pro_{z}(Z_{q\Delta}\in I)>0$.\\
To generalize from open Borel sets to Borel sets, we need some regularity of the function $z\mapsto\Pro_{z}(Z_{\Delta}\in A)$ for $A\in\mathcal{B}(E')$. 
We compute this probability by distinguishing the trajectories with the number of jumps $J(\Delta)$ occurring on $[0,\Delta]$, 
then for all $z\in E'$ and $A\in\mathcal{B}(E')$,
\bean
\label{sum}
\Pro_{z}\bigl( Z_{\Delta}\in A\bigr) &=\sum_{k=0}^{\infty} \Pro_{z}\bigl( Z_{\Delta}\in A \text{ and } J(\Delta) =k\bigr).
\eean
We recall that the sequence of jump times of the prey population is denoted by $(T_k)_{k\in \N}$ and that $\Pro_{z}(T_1\ge t)=e^{-\Theta(z,t)}$ where $\Theta(z,t)=\int_0^t \theta(\phi(z,s))ds $ and the total jump rate $\theta(z)$ is given by \eqref{theta}.\\
The first term of \eqref{sum} handles trajectories where no jump occurs. It is given by
$$
\Pro_{z}\bigl( Z_{\Delta}\in A \text{ and } J(\Delta) =0\bigr)= e^{-\Theta(z,\Delta)ds} \un_A(\phi(z,\Delta)).
$$
This function is not continuous in $z$ since the indicator function $\un_{A}$ is not continuous and the total flow $\phi(z,\Delta)$ is continuous.\\
The idea is then to bound from below $\Pro_{z}(Z_{\Delta}\in A)$ by a continuous function (see Chapter 6 of \cite{meyntweedie} and \cite{bayer2011existence}). In the sequel we consider
\bea
\Pro_{z}\bigl( Z_{\Delta}\in A\bigr) 
&\ge \Pro_{z}\bigl( Z_{\Delta}\in A \text{ and } J(\Delta) =1\bigr),
\eea
and prove that for any $A\in\mathcal{B}(E')$, the function $z\mapsto T(z,A)=\Pro_{z}\bigl( Z_{\Delta}\in A\text{ and } J(\Delta) =1\bigr)$ is continuous on $E'$. The continuity will derive from the fact that the law of first jump time has a density with respect to the Lebesgue measure.\\
Indeed, for any $z=(n,h)\in E'$
\bea
\Pro_{z}&\bigl( Z_{\Delta}\in A \text{ and } J(\Delta) =1\bigr)=\Pro_{z}\bigl( Z_{\Delta}\in A \text{ and } T_1\le \Delta <T_2\bigr)\\
&=\int_0^{\Delta}\un_A \Bigl(n+1, \phi_{n+1}\bigl( \phi_n(h,s) ,\Delta-s \bigr)\Bigr) bn e^{- \Theta(n,h,s)}e^{\Theta(n+1,\phi_n(h,s),\Delta-s)} ds\\ 
&+\int_0^{\Delta}\un_A \Bigl(n-1, \phi_{n-1}\bigl( \phi_n(h,s) ,\Delta-s \bigr)\Bigr) n(d+cn+B\phi_n(h,s))\un_{n\ge2}\\
&\quad\quad \times e^{-\Theta(n,h,s)}e^{\Theta(n-1,\phi_n(h,s),\Delta-s)} ds.
\eea
The first integral corresponds to the event where a birth occurs at $T_1$ while the second interval to the event where a death happens at $T_1$. In the sequel, we consider the first integral. The study on the second integral is very similar and will not be detailed.
The predator density at time $\Delta$ conditioned on the fact that only one jump happens on $[0,\Delta]$ and is a birth occurring at time $s\in[0,\Delta]$ is given by 
\be
g_{(n,h,\Delta)}(s)=\phi_{n+1}\bigl( \phi_n(h,s) ,\Delta-s \bigr).
\ee
We note that for any $s$ and $\Delta$, the application $(n,h)\mapsto g_{(n,h,\Delta)}(s)$ is continuously differentiable.
To perform the change of variable $y=g_{(n,h,\Delta)}(s)$ in the previous integral, we have to verify that $\frac{d}{ds}g_{(n,h,\Delta)}(s)$ does not vanish.
\be
\frac{d}{ds}g_{(n,h,\Delta)}(s)= \partial_2 \phi_n(h,s) \cdot \partial_1 \phi_{n+1}\bigl( \phi_n(h,s) ,\Delta-s \bigr) -\partial_2\phi_{n+1}\bigl( \phi_n(h,s) ,\Delta-s \bigr).
\ee
We recall that $\phi_n$ is the flow associated with \eqref{eqdiff-phin}, then
$$\partial_2 \phi_n(h,s)=\phi_n(h,s) (rBn-D-C\phi_n(h,s)).
$$
From the exact expression \eqref{phi_n} we obtain that
$$\partial_1 \phi_n(h,s)=\frac{\phi_n(h,s) }{h\bigl(1+\frac{hC}{rBn-D}\bigl( e^{(rBn-D)s}-1 \bigr)\bigr)}.
$$
Then an easy calculation using \eqref{phi_n} leads to
\bea
\frac{d}{ds}g_{(n,h,\Delta)}(s)= &-rB\frac{g_{(n,h,\Delta)}(s)}{1+\frac{\phi_n(h,s)C}{rB(n+1)-D}\bigl( e^{(rB(n+1)-D)(\Delta-s)}-1 \bigr) }<0.
\eea
Let us finally remark that 
$$g_{(n,h,\Delta)}(0)=\phi_{n+1}(h,\Delta)\quad\text{and}\quad g_{(n,h,\Delta)}(\Delta)= \phi_n(h,\Delta),$$ 
hence
\bean
\label{int1}
\int_0^{\Delta}\un_A& \Bigl(n+1, \phi_{n+1}\bigl( \phi_n(h,s) ,\Delta-s \bigr)\Bigr) bn e^{- \Theta(n,h,s)}e^{\Theta(n+1,\phi_n(h,s),\Delta-s)} ds\\
&=\int_{\phi_n(h,\Delta)}^{\phi_{n+1}(h,\Delta)}\un_A \Bigl(n+1,y\Bigr) f(n,h,\Delta,y)dy,
\eean
where
\be
f(n,h,\Delta,y)=bn\frac{ e^{- \Theta\bigl( n,h, g_{(n,h,\Delta)}^{-1}(y) \bigr)}e^{\Theta\bigl( n+1,\phi_n\bigl(h, g_{(n,h,\Delta)}^{-1}(y)\bigr),\Delta-g_{(n,h,\Delta)}^{-1}(y)\bigr)} }
{\left|\frac{d}{ds}g_{(n,h,\Delta)}(s )\Bigr|_{s=g_{(n,h,\Delta)}^{-1}(y) }\right|}.
\ee
The function of $(n,h)$ given defined by \eqref{int1} is continuous since the upper and the lower bounds of the integral are continuous functions of $(n,h)$ on $E'$ and the integrand is continuous in $(n,h)$ on $E'$ and locally bounded.\\
To conclude with the irreducibility of $(Z_{k\Delta})_{k\in \N}$, we fix a point $z_0\in E'$ and remark that the measure $T(z_0,\cdot)=\Pro_{z}\bigl( Z_{\Delta}\in \cdot\text{ and } J(\Delta) =1\bigr)$ is non degenerate since $T(z_0,E')>0$.\\
For any $A\in\mathcal{B}(E')$ such that $T(z_0,A)>0$, there exists, by continuity of $T$, an open neighborhood $\mathcal{O}$ of $z_0$ such that $\forall z\in\mathcal{O}$, $T(z,A)>T(z_0,A)/2$. Moreover we deduce from Theorem \ref{thm:irr+} that for any initial condition $(n,h)$ there exists $q\in\N$ such that 
$\Pro_{(n,h)}(Z_{q\Delta}\in\mathcal{O})>0.$
Then it derives from the Markov property at time $q\Delta$ and from the properties of the kernel $T$ that
\bea
\Pro_{(n,h)}\bigl(Z_{(q+1)\Delta}\in A\bigr)&\ge \E_{(n,h)}\bigl(\un_{Z_{(q+1)\Delta}\in A} \E_{Z_{q\Delta}} \bigl(\un_{ Z_{\Delta}\in\mathcal{O}}\bigr)\bigr)\\
&\ge \Pro_{(n,h)}\bigl(Z_{q\Delta}\in\mathcal{O}\bigr)\frac{T\bigl(z_0,A\bigr)}{2}>0
\eea
Hence, $(Z_{k\Delta})_{k\in\N}$ is irreducible with respect to the measure $T(z_0,\cdot)$.
\end{proof}

\begin{rem}
Here we chose to study the behavior of skeletons of the process in order to derive the exponential ergodicity of the process. This standard method has already been used in the context of PDMP (see for example \cite{last2004ergodicity,bayer2011existence}). Costa and Dufour \cite{dufour1999stability, costa2008stability}  developed a different approach  based on an embedded Markov chain whose long time behavior is equivalent to those of the PDMP. For the prey-predator process $Z$ it seems to us more natural to study the continuous time trajectories since comparisons where possible. Moreover in future work, we aim at deriving from these trajectory constructions more quantitative convergence speeds using recently developed coupling methods for PDMP (e.g. \cite{bouguet2013quantitative,fontbona2012quantitative}).
\end{rem}

\begin{rem}
The results of this section can be extended to the model including migration introduced in Remark \ref{rem_migr}. Indeed, using the very same trajectory construction, one can prove that the process $(N^{(m)},H^{(m)})$ is irreducible on $E^{(m)}$ for the Lebesgue measure $\sigma$. Moreover, the infinitesimal generator \eqref{gene_migr} satisfies a Foster-Lyapunov inequality with the same Lyapunov function. Therefore  $(N^{(m)},H^{(m)})$ admits a unique probability invariant measure. Finally, by adapting the proof of Theorem \ref{thm:exp-ergodicity_intro} we get that it is exponentially ergodic.
\end{rem}

\section{Re-scaling the predator dynamics}
\label{sec:averaged}
We introduce a new parameter $\varepsilon$ which rescales the predator dynamics and illustrates the biological assumption that the predator mass is almost negligible comparing to the prey mass. There exists an important literature about the metabolic theory which describes the relationships between mass and metabolic characteristics of living individuals (see among others \cite{Damuth81,brown2004toward}).
Following this theory, the demographic parameters of individuals increase when their mass decreases.  Here we simplify these relationships by assuming that the predator parameters vary as $1/\varepsilon$.\\
For any $\varepsilon\in(0,1]$, we consider the community process $\Ze=(\ne,\he)$ where for all $t\ge0$, 
\ben
\label{eq-diff-eps}
\frac{d}{dt}\he_t=\frac{\he_t}{\varepsilon}(rB\ne_t-D-C\he_t)
\een
and the dynamics of $\ne$ is given by the jump mechanism \eqref{sauts} associated with the predator population $\he$.
The process studied in the previous sections corresponds to $\varepsilon=1$ or to the parameters  $r^{\varepsilon}=r/\varepsilon$, $D^{\varepsilon}=D/\varepsilon$ and $C^{\varepsilon}=C/\varepsilon$.\\
This scaling changes the time scale of the predator flow. If $\phi_n^{\varepsilon}$ is the flow associated with \eqref{eq-diff-eps} then 
\ben
\label{lien-flot}
\phi^{\varepsilon}_n(h,t)=\phi^1_n(h,\frac{t}{\varepsilon}), \quad \forall t\ge0, \forall (n,h)\in E.
\een

\subsection{Convergence toward an averaged process}
In the sequel we study the limit as $\varepsilon$ tends to $0$ of the sequence $\Ze$ in $\mathbb{D}([0,T],E)$. The prey-predator process is a slow-fast system. As $\varepsilon$ diminishes, the predator process converges faster to its equilibrium between the jumps of the prey population. The slow dynamics of the prey population is then averaged on the predator equilibria.
\smallskip\\
We first give the expression of the infinitesimal generator of the process $\Ze$ defined above: $\forall f\in \Ecal$ and $(n,h)\in E$
\bea
\mathcal{A}^{\varepsilon}f(n,h)=&\partial_2 f(n,h)\frac{h(rBn-D-Ch)}{\varepsilon}+bn(f(n+1,h)-f(n,h))\\
&+n(d+cn+Bh)(f(n-1,h)-f(n,h))\un_{n\ge2}.
\eea

To carry out the limit as $\varepsilon\to0$, we use the fact that the convergence speed of the flow $\phi_n^{\varepsilon}$ to its equilibrium $h^*_n$ is uniform in $\varepsilon$. This is only true if the number of predators remains bounded from below by some strictly positive constant. To this aim, we make the following assumptions
\bean
\label{hyp2}
&i)\quad rB-D>0.\text{ This implies that the predator population cannot }\\&\quad\text{become extinct} \\
&ii)\quad \text{We restrict ourselves to initial conditions in the set }\\&\quad E'=\{1,\cdots\}\times [h^*_1,\infty).
\eean
The state space $E'$ is stable for the prey-predator dynamics and is the support of the irreducibility measure $\sigma'$ introduced in Theorem \ref{thm:phiirr} ii).
\begin{prop}
\label{prop:unif-conv}
Under Assumption \eqref{hyp2}, there exists $\delta>0$ such that for all $(n,h)\in E'$,
\ben
\label{unif-conv}
|\phi_n^{\varepsilon}(h,t)-h^*_{n}|\le|h-h^*_{n}|e^{-\delta\frac{t}{\varepsilon}}
\een
\end{prop}
\begin{proof}
Let us remark that thanks to \eqref{lien-flot}, it is sufficient to prove the result for $\varepsilon=1$.\\
Using \eqref{phi_n}, an easy computation leads to
\bea
|\phi_n(h,t)-h^*_{n}| =\left| \frac{h-h^*_{n}}{1+\frac{h}{h^*_{n}}(e^{-t(rBn-D)}-1)}\right |, \quad \forall (n,h)\in E' \text{ and } t\ge0.
\eea
Therefore it is enough to find $\delta>0$ satisfying $\forall (n,h)\in E'$  and $ t\ge0$:
$$
e^{\delta t} \le 1+\frac{h}{h^*_{n}}(e^{-t(rBn-D)}-1).
$$
The function $t\mapsto 1+\frac{h}{h^*_{n}}(e^{-t(rBn-D)}-1)-e^{\delta t}$ equals $0$ for $t=0$, and increases as soon as 
$hC\ge \delta$ and $rBn-D\ge\delta$. Since $h\ge h^*_1>0$ we choose $\delta<Ch^*_1 $ to obtain \eqref{unif-conv}.
\end{proof}
\bigskip

Following Kurtz \cite{kurtz1992averaging}, we introduce the predator occupation measure
\ben
\label{occupation}
\Gamma^{\varepsilon}([0,t],A)=\int_0^t \un_A(\he_s)ds, \quad \forall t\ge0 \text{ and }A\in\mathcal{B}(\R^+).
\een 
This random measure belongs to the set $\mathcal{M}_m(\R_+)$ of measures $\mu$ on $\R_+\times\R_+$ such that $\mu([0,t]\times\R_+)=t$, $\forall t\ge0$. For any $t\ge0$, we denote by $\mathcal{M}_m^t(\R_+)$ the set of the measures $\mu\in \mathcal{M}_m(\R_+)$ restricted to $[0,t]\times\R_+$.\\
In the sequel we prove using the averaging method developed in \cite{kurtz1992averaging} that the sequence $(\ne,\Gamma^{\varepsilon})$ converges in law. This method allows us to avoid the difficulties related to the fast convergence of the predator flow to its equilibrium.
\begin{theorem}
 \label{thm:cv-moy}
Fix $ T>0$ and assume \eqref{hyp2}. We suppose that the sequence of initial conditions $(\Ze_0)_{0<\varepsilon\le1}$ converges to $\widebar{Z}_0$ in law
and moreover that
\ben
\label{hyp:moment2}\sup_{0<\varepsilon\le1}\E((\ne_0)^4)<\infty
, \quad \sup_{0<\varepsilon\le1}\E((\he_0)^4)<\infty.
\een
Then the sequence $\bigl(\ne,\Gamma^{\varepsilon}\bigr)$ converges in law toward $(\av,\Gamma)$ in $\mathbb{D}([0,T],\N)\times \mathcal{M}_m^T(\R_+)$  as $\varepsilon\to 0$. \\
The process $\av$ is a pure jump process on $\N^*$ whose infinitesimal generator is well defined for every measurable and bounded function $f:\N^*\to\R$ by
\ben
\label{gen-lim}
\mathcal{L} f(n)=(f(n+1)-f(n))bn+(f(n-1)-f(n)) n(d+cn+Bh^*_n)\un_{n\ge2}.\een
Moreover, the limiting measure is defined by $\Gamma(ds \times dy)=\delta_{h^*_{\av_s}}(dy)ds$.
\end{theorem}
We say that $\av$ is an averaged process since it behaves as if the predator density is constant at its equilibrium.
Let us consider the specific case where $D=0$. In this case, the averaged prey population $\av$ evolves almost as a logistic birth and death process with individual birth rate $b$ and individual death rate $d+\widetilde{c}n$ where $\widetilde{c}=c+rB^2/C$. The logistic parameter $\widetilde{c}$ corresponds to the apparent competition pressure (see \cite{Amstrong80}): it takes into account both the prey competition $c$ and the effect of predation.
\begin{proof}
We divide the proof in several steps. The first three steps are devoted to the convergence of the prey population process. We use a standard compactness-identification method inspired from  Genadot \cite{genadot2014multi}. In the fourth step, we use the averaging method developed by Kurtz \cite{kurtz1992averaging} to prove the convergence of the predator occupation measures.
\medskip\\
\noindent\textsc{Step 1:} We prove that there exists an unique (in law) solution to the martingale problem associated to \eqref{gen-lim}: for every measurable and bounded function $f:\N^*\to\R$
\be
f(\av_t)-f(\av_0)-\int_0^t \mathcal{L}f(\av_s) ds
\ee
is a martingale.\\
It derives from the representation Theorem 3.2 in \cite{kurtz2011equivalence} and a localization argument (see Theorem 4.6.3 in \cite{Ethier&Kurtz}) that the uniqueness of the solution of this martingale problem is equivalent to the uniqueness in law of the solution $\av$ of the following stochastic differential equation provided that $\sup_{s\le t} \av_s<\infty$ a.s. for any $t\ge0$:
\bean
\label{sto}
\av_t=\av_0+\int_0^t\int_{R_+}& \un_{u\le b\av_{s-}}\\&-\un_{b\av_{s-}<u\le \av_{s-}(b+d+c\av_{s-}+Bh^*_{\av_{s-}})} Q(ds,du),
\eean
where $Q$ is a Poisson point measure on $(\R_+)^2$ with intensity the product of Lebesgue measure $ds du$.
The uniqueness of the weak solution of \eqref{sto} can be adapted from \cite{FM04}. Moreover, if $\E(\av_0)<+\infty$, then this solution is well defined on $\R_+$ and $\sup_{s\le t} \av_s<\infty$ a.s. for any $t\ge0$.
\medskip\\
\noindent\textsc{Step 2:} Tightness of $(\ne_t,0\le t\le T)$\\
Similarly to the proof of Theorem \ref{thm:1}, we construct a pure birth process $X$ independent from $\he_0$ and thus from $\varepsilon$, which dominates the prey population. Then we deduce from \eqref{hyp:moment2} that $\forall z\in E'$
\ben
\label{moment-eps}
\sup_{0<\varepsilon\le1}\E_z( \sup_{s\in[0,T]} (\ne_s)^4) \le \E_{z}( \sup_{s\in[0,T]} (X_s)^4)<+\infty.
\een
Let us now fix $\eta,\delta>0$ and consider stopping times $\sigma,\tau$ such that $\sigma\le\tau\le (\sigma+\delta)\wedge \tau$.
Using the trajectory's construction \eqref{eq:poisson}, we write
$$\ne_t=\ne_0+\int_0^t \ne_s(b-(d+c \ne_s +B\he_s)\un_{\ne_s\ge2}) ds+ M^{\varepsilon}_t,$$
where $M^{\varepsilon}_t$ is a pure jump martingale with quadratic variation
$$
\langle M^{\varepsilon}\rangle_t=\int_0^t \ne_s(b+(d+c \ne_s +B\he_s)\un_{\ne_s\ge2}) ds.
$$
Hence
\bea
\E\Bigl( \bigl(& \ne_{\tau}-\ne_{\sigma}\bigr)^2\Bigr)\\
&= \E\Bigl( \bigl(\int_{\sigma}^{\tau} \ne_s(b-(d+c \ne_s +B\he_s)\un_{\ne_s\ge2}) ds+ M^{\varepsilon}_{\tau}-M^{\varepsilon}_{\sigma} \bigr)^2  \Bigr)\\
&\le 2\Bigl[\E\Bigl( \bigl(\int_{\sigma}^{\tau} \ne_s(b+d+c \ne_s +B\he_s) ds\bigr)^2  \Bigr) + \E\Bigl( \bigl( M^{\varepsilon}_{\tau}-M^{\varepsilon}_{\sigma} \bigr)^2  \Bigr)\Bigr].
\eea
We deduce from \eqref{majoration-phi_n} that the first term is bounded above by
\bea
\E\Bigl( \bigl(\int_{\sigma}^{\tau} \ne_s(b+d+c \ne_s +B(\he_0+\frac{rB}{C}\sup_{u\in[0,s]  }\ne_u) ds\bigr)^2  \Bigr).
\eea
Therefore, using \eqref{hyp:moment2} and \eqref{moment-eps}, there exists a constant $C_T>0$ such that
\bea
\E\Bigl( \bigl(\int_{\sigma}^{\tau} \ne_s(b+d+c \ne_s +B\he_s) ds\bigr)^2  \Bigr)\le C_T \delta^2.
\eea
For the second term 
\bea
\E\Bigl( \bigl( M^{\varepsilon}_{\tau}-M^{\varepsilon}_{\sigma} \bigr)^2  \Bigr)&= \E\Bigl(\int_{\sigma}^{\tau} \ne_s(b+(d+c \ne_s +B\he_s)\un_{\ne_s\ge2}) ds \Bigr)\\
&\le \E\Bigl(\int_{\sigma}^{\tau} \ne_s(b+d+c \ne_s +B(\he_0+\frac{rB}{C}\sup_{u\in[0,s]  }\ne_u)) ds \Bigr),
\eea
where the last inequality derives from \eqref{majoration-phi_n}.
Combining this two inequalities \eqref{hyp:moment2} and \eqref{moment-eps}, we deduce that
\bea
\E\Bigl( \bigl( \ne_{\tau}-\ne_{\sigma}\bigr)^2\Bigr)
&\le \delta C_T,
\eea
for some constant $C_T>0$ independent of $\varepsilon$ which leads to the tightness of the laws of  $(\ne_t,t\in [0,T])$ in $\mathbb{D}([0,T],\N)$ using Aldous tightness criterion \cite{aldous1978stopping}. 
\bigskip\\
\noindent\textsc{Step 3:} Identification of the limit\\
Let us consider a sub-sequence of $(\ne_t,0\le t\le T)$ (still denoted $\ne$) which converges in law when $\varepsilon\to0$ toward a process {$(\av_t, 0\le t \le T)$}. In the sequel we prove that $\av$ is the unique solution of the martingale problem associated to \eqref{gen-lim}.\\
We consider bounded functions $f$, $g_1,\dots, g_k$ on $\N^*$ and times $0\le t_1<\dots<t_k<t<t+s\le T$. We deduce from Theorem \ref{thm:1} that
\bean
\label{esp-f-ne}
\E&\Bigl[\Bigl(f(\ne_{t+s})-f(\ne_t) -\int_{t}^{t+s} \mathcal{L}f(\ne_u)du\Bigr) \prod_{i=1}^k g_i(\ne_{t_i}) \Bigr]\\
+&\E\Bigl[\Bigl(\int_{t}^{t+s}(f(\ne_u-1)-f(\ne_u)) \ne_uB(\he_u-h^*_{\ne_u})\un_{\ne_u\ge2} du\Bigr)\prod_{i=1}^k g_i(\ne_{t_i})\Bigr] =0.
\eean
From the convergence in law of $\ne$ to $\av$ and \eqref{moment-eps}, the first term converges as $\varepsilon\to 0$ to
\be
\E\Bigl[\Bigl(f(\av_{t+s})-f(\av_t) -\int_{t}^{t+s} \mathcal{L}f(\av_u)du\Bigr) \prod_{i=1}^k g_i(\av_{t_i}) \Bigr].
\ee
Let us prove that
\ben
\label{cv1}\lim_{\varepsilon\to0} \E\Bigl[\Bigl(\int_{t}^{t+s}(f(\ne_u-1)-f(\ne_u)) \ne_uB(\he_u-h^*_{\ne_u})\un_{\ne_u\ge2} du\Bigr)\prod_{i=1}^k g_i(\ne_{t_i})\Bigr]=0.
\een
We first remark that 
\bea\E\Bigl[\Bigl(&\int_{t}^{t+s}\bigl(f(\ne_u-1)-f(\ne_u)\bigr) \ne_uB(\he_u-h^*_{\ne_u})\un_{\ne_u\ge2} du\Bigr)\prod_{i=1}^k g_i(\ne_{t_i})\Bigr]\\
&\le \prod_{i=1}^k\norm{g_i}_{\infty} \E\Bigl[\int_{t}^{t+s}\lvert f(\ne_u-1)-f(\ne_u)\lvert\ne_uB \lvert\he_u-h^*_{\ne_u}-\lvert \un_{\ne_u\ge2} du\Bigr]
\eea
We split the integral according to the sequence of jump times $(T^{\varepsilon}_k,k\in\N)$ of the prey population:
\bea
\Bigl|\E\Bigl(\int_{t}^{t+s}& \lvert f(\ne_u-1)-f(\ne_u)\lvert \ne_uB |\he_u-h^*_{\ne_u}| \un_{\ne_u\ge2} du\Bigr)\Bigr|\\
&\le 2\norm{f}_{\infty} B \sum_{k=0}^{\infty} \sum_{n=1}^{\infty} \E\Bigl(\un_{n}(\ne_{T^{\varepsilon}_k}) n \int_{T^{\varepsilon}_k}^{T^{\varepsilon}_{k+1}\wedge (t +s)}|h^*_{\ne_u}-\he_u| du\Bigr).
\eea
Since on the event $\{\ne_{T^{\varepsilon}_k}=n\}$, for any $u\in [T^{\varepsilon}_{k}, T^{\varepsilon}_{k+1}\wedge (t+s)]$, the predator density is given by $\he_u=\phi_n^{\varepsilon}(\he_{T^{\varepsilon}_k}, u-T^{\varepsilon}_k)$, we derive from Proposition \ref{prop:unif-conv} that
\be
|h^*_{\ne_u}-\he_u|\le |h^*_{\ne_u}-\he_{T^{\varepsilon}_u}|e^{-\frac{\delta(u-T^{\varepsilon}_k)}{\varepsilon}}.
\ee
Hence
\bea
\Bigl|\E\Bigl(&\int_{t}^{t+s}(f(\ne_u-1)-f(\ne_u)) \ne_uB(\he_u-h^*_{\ne_u})\un_{\ne_u\ge2} du\Bigr)\Bigr|\\
&\le 2\norm{f}_{\infty} B \sum_{k=0}^{\infty} \sum_{n=1}^{\infty} \E\Bigl(\un_{n}(\ne_{T^{\varepsilon}_k})  n \int_{T^{\varepsilon}_k}^{T^{\varepsilon}_{k+1}\wedge (t +s)} |h^*_{n}-\he_{T^{\varepsilon}_k}| {e^{- \frac{\delta(u-T^{\varepsilon}_k)}{\varepsilon}  }} du\Bigr).
\eea
From \eqref{majoration-phi_n}, we obtain that for every $T>0$
\ben
\label{maj-etoile}
\sup_{u\in[0,T]} \he_u\le H_++\frac{rB}{C} \sup_{u\in [0,T]} \ne_u <+\infty.
\een
Since $h^*_n\le rBn/C$, we deduce that for all $u\in[T^{\varepsilon}_k,T^{\varepsilon}_{k+1}\wedge (t+s)]$
\be
|h^*_{\ne_{T^{\varepsilon}_k}}-\he_{T^{\varepsilon}_k}|\le H_+ +2\frac{rB}{C}\sup_{u\in[0,T]} \ne_u.
\ee
With a change of variable $u\to u-T^{\varepsilon}_k$, we obtain that
\bea
\int_{T^{\varepsilon}_k}^{T^{\varepsilon}_{k+1}\wedge (t+s) }{e^{- \frac{\delta(u-T^{\varepsilon}_k)}{\varepsilon}  }} du
&= \int_0^{T^{\varepsilon}_{k+1}\wedge (t +s) -T^{\varepsilon}_k} \exp(-\frac{\delta u}{\varepsilon})du\\
& \le \int_{0}^{t+s }\exp(-\frac{\delta u}{\varepsilon})du\un_{T^{\varepsilon}_k\le t+s}\\
&\le \frac{\varepsilon}{\delta}\un_{T^{\varepsilon}_k\le t+s}.
\eea
Combining these two inequalities,  we have that
\bea
\Bigl|\E\Bigl(\int_{t}^{t+s}&(f(\ne_u-1)-f(\ne_u)) \ne_uB(\he_u-h^*_{\ne_u})\un_{\ne_u\ge2} du\Bigr)\Bigr|\\
&\le 2\norm{f}_{\infty} B\quad \E\Bigl(\bigl(H_+ +2\frac{rB}{C}\sup_{u\in[0,T]} \ne_u \bigr)\sup_{u\in[0,T]} \ne_u \sum_{k=1}^{\infty}\un_{T_k^{\varepsilon}\le t+s}\Bigr)
\frac{\varepsilon}{\delta}.
\eea
It remains to bound the expectation independently of $\varepsilon$. To this aim, we denote by $J_t^{\varepsilon}$ the number of jumps before time $t$. 
By neglecting the non-positive terms, we deduce from the trajectorial construction \eqref{eq:poisson} that
\bea
\sup_{t\in[0,T]}\ne_t J_T^{\varepsilon} &\le \int_0^T\int_{\R_+} \bigl( \ne_{s-}+J_{s-}^{\varepsilon}+1\bigr) \un_{u\le b\ne_{s-} }Q_1(ds,du)\\
&+\int_0^T\int_{\R_+} \bigl( \ne_{s_-}\bigr) \un_{u\le \ne_{s-}(d+c\ne_{s-}+B\he_{s-})} Q_2(ds,du).
\eea
Then, we choose $S\ge T$ and taking expectations we deduce from \eqref{hyp:moment2} and \eqref{moment-eps} that there exists a positive constant $C_S$ such that
\bea
\E\Bigl(\sup_{t\in[0,T]}\ne_t J_T^{\varepsilon}\Bigr) &\le C_S+\E\Bigl( \int_0^T b J_{t}^{\varepsilon}\sup_{s\in[0,t]}\ne_{s} dt\Bigr).
\eea 
We conclude using Gronwall lemma that $\E\Bigl(\sup_{t\in[0,T]}\ne_t J_T^{\varepsilon}\Bigr) <\infty$.
A very similar computation leads to 
\bea
\E\Bigl(\sup_{t\in[0,T]}(\ne_t)^2 J_T^{\varepsilon}\Bigr)\le \E\Bigl(\int_0^T& 2b(\ne_s)^2J_s^{\varepsilon} +b\ne_sJ_s^{\varepsilon}+b\ne_s(1+\ne_s)\\
& + (\ne_s-1)^2\ne_s(d+c\ne_s+B\he_s) ds\Bigr) .
\eea
We conclude with \eqref{hyp:moment2}, \eqref{moment-eps} and Gronwall lemma that 
$$\E\Bigl(\sup_{t\in[0,T]}(\ne_t)^2 J_T^{\varepsilon}\Bigr) <\infty.$$
Therefore, we deduce that there exists a constant $C_{S,T}>0$ such that 
\be
\Bigl|\E\Bigl(\int_{t}^{t+s}(f(\ne_u-1)-f(\ne_u)) \ne_uB(\he_u-h^*_{\ne_u})\un_{\ne_u\ge2} du\Bigr)\Bigr|\le C_{S,T} \frac{\varepsilon}{\delta}.
\ee
Thus, \eqref{cv1} is verified.\\
Therefore
\bea
\E\Bigl[\Bigl(&f(\av_{t+s})-f(\av_t) -\int_{t}^{t+s} \mathcal{L}f(\av_u)du\Bigr) \prod_{i=1}^k g_i(\av_{t_i}) \Bigr]=0,
\eea
and thus $\av$ is a solution of the martingale problem associated with $\mathcal{L}$.
From \eqref{hyp:moment2} we deduce that $\E(\av_0)<\infty$ and then from the first step, $\ne$ converge in law to the unique solution of the martingale problem associated with $\mathcal{L}$ which is a birth and death process which jumps from $n\to n+1$ at rate $bn$ and from $n\to n-1$ at rate $n(d+c n +Bh^*_{n})\un_{n\ge2}$.
\bigskip\\
\noindent\textsc{Step 4:} Limit behavior of the predator population size\\
We prove the weak convergence of the sequence of occupation measures $\Gamma^{\varepsilon}$ introduced in \eqref{occupation}.
From, Lemma 1.3 in \cite{kurtz1992averaging}, the sequence of the laws of $\Gamma^{\varepsilon}$ is tight in the set $\mathcal{M}_m^T(\R_+)$ if 
for any $\delta>0$ and $t\in[0,T]$ there exists a compact $K\subset \R_+$ such that
\be
\inf_{\varepsilon}\E\bigl(\Gamma^{\varepsilon}([0,t]\times K)\bigr) \ge (1-\delta) t.
\ee
From \eqref{maj-etoile}, we deduce that the family
$\{\he_s,\varepsilon\in(0,1],s\in[0,T]\}$ is relatively compact, and thus the sequence of the laws of $(\Gamma^{\varepsilon})_{\varepsilon\in(0,1]}$ is tight.\\
Hence, the pair $(\ne,\Gamma^{\varepsilon})$ is tight and we consider a sub-sequence, still denoted $(\ne,\Gamma^{\varepsilon})$, that converges toward $(\av,\Gamma)$.

Since $\Gamma^{\varepsilon}([0,t],\R^+)=t$, for all $\varepsilon\in(0,1)$ and $t\in[0,T]$, we deduce (see Lemma 1.4 in \cite{kurtz1992averaging}) that there exists a process $\gamma_s$ taking values in the set of probability measures on $\R^+$, measurable with respect to $(\omega,s)$, such that for every measurable and bounded function $h$ on $[0,T]\times \R_+$
\be
\int_0^T\int_{\R_+}h(s,y)\Gamma(ds\times dy)=\int_0^T\int_{\R_+}h(s,y)\gamma_s(dy) ds.
\ee
Using \eqref{moment-eps}, we deduce from Theorem 2.1  in \cite{kurtz1992averaging} that for every function $f$ on $\N^*$ continuous and bounded and $t\le T$
\bean 
\label{eq1}
f(\av_t)- &f(\av_0)-\int_0^t{\int_{R_+}} \Bigl[(f(\av_s+1)-f(\av_s))b\av_s\\
&+ (f(\av_s-1)-f(\av_s))\av_s (d+c\av_s +B y)\un_{\av_s\ge2}\Bigr] \gamma_s(dy) ds.
\eean
is a martingale.\\
Concerning the predator dynamics, for every function $g\in \mathcal{C}^1_b(R_+)$, the process
$$\varepsilon M^{\varepsilon}_g(t)=\varepsilon g(\he_t)-\varepsilon g(\he_0)-\int_0^t g'(\he_s) \he_s\bigl(rB\ne_s-D-C\he_s\bigr)ds;
$$
is a martingale. Using \eqref{moment-eps}, we prove easily that $\bigl(\varepsilon M^{\varepsilon}_g(t)\bigr)_{0<\varepsilon\le 1,t\in[0,T]}$ is uniformly integrable and that
$$\lim_{\varepsilon\to0} \E( |\varepsilon M^{\varepsilon}_g(t)-\widetilde{M}_g(t)|)=0,$$
where $\widetilde{M}_g(t)=\int_0^t \int_{\R_+}g'(y) y\bigl(rB\av_s-D-Cy\bigr)\gamma_s(dy)ds$. Therefore we deduce from the uniform integrability of $\varepsilon M^{\varepsilon}_g$ that $\widetilde{M}_g(t)$ is a martingale. Since it is also a continuous and finite variation process, it must be null. Thus 
\ben 
\label{eq2}
\int_{\R^+} g'(y) y(rB\av_t-D-Cy)  \gamma_t(dy)=0,\quad \quad \text{for } dt-\text{almost every }t\in[0,T].
\een

We recall that the infinitesimal generator of $\av$ is given by \eqref{gen-lim}, then by identification in \eqref{eq1} we deduce that
for all $t\in[0,T]$, 
\ben
\label{fin}
\int_{\R^+} y \gamma_t(dy) = h^*_{\av_t}.
\een
{Then we apply \eqref{eq2} to the function $g(y)=\int_y^{\infty} f(u)du$,} where $f$ is continuous on a compact of $\R_+$, 
\be
\int_{\R^+} f(y)y(rB\av_t-D-Cy)\gamma_t(dy)=0, \quad \text{for } dt-\text{almost every }t\in[0,T].
\ee
From Riesz Theorem, the measure $\mu(\cdot)=\int_\cdot y(rB\av_t-D-Cy)\gamma_t(dy)$ on $\R_+$ is null for almost every $t\in[0,T]$.
Then, for almost every $t\in[0,T]$, $\gamma_t$ only charges $0$ and $h^*_{\av_t}$. Finally using \eqref{fin} we conclude that $\gamma_t(dy)=\delta_{h^*_{\av_t}}(dy)$ for $dt-$ almost every $t\in[0,T]$.

\end{proof}

\begin{rem}The proof of Theorem \ref{thm:cv-moy} relies on the uniform convergence of the flows $\phi_n(h,t)$ to their equilibrium $h^*_n$ as $t\to\infty$, $\forall (n,h)\in E'$. It is not possible to extend this reasoning to the process $(N^{(m)},H^{(m)})$ including migration introduced in Remark \ref{rem_migr}, since its state space is $\N\times\R_+$. However, we can still prove using the same method that the accelerated sequence $(N^{(m),\varepsilon},\Gamma^{(m),\varepsilon})$ is tight (with obvious notations). The difficulty to extend Theorem \ref{thm:cv-moy} to the migration case lies in the identification of the limiting values.
\end{rem}
\subsection{Long time behavior of the averaged process}
We are interested in the long time behavior of the averaged prey population $\av$.
\begin{prop}
\label{prop:pib}
The process $\av$ is positive recurrent on ${\N^*}$ 
and converges toward its unique invariant probability measure $\widebar{\mu}=\sum_{n=1}^{\infty} \mu_n \delta_{n}$ which satisfies the system
\bean
\label{system}
\forall n\ge2,\quad &\mu_n =\frac{b^{n-1}}{n\quad \Pi_{i=2}^n (d+ci+Bh^*_i)} \mu_1\\
\eean
 \end{prop}
The proof is very classical and left to the reader (see for example \cite{allen2010introduction}, p.216). Moreover, we are interested in the process $\widebar{Z}=(\av,h^*_{\av})$ which represents the averaged prey and predator populations. We obtain immediately the form of its invariant distribution since $h^*_{\av}$ is a function of $\av$.
\begin{cor}
\label{cor:unik}
The process $\widebar{Z}=(\av,h^*_{\av})$ admits a unique invariant probability measure $\pib$ given by
\ben
\label{pib}
\pib=\sum_{n\ge1}\mu_n\delta_{(n,h^*_n)},
\een
where the sequence $\mu_n$ satisfies \eqref{system}
\end{cor}
The expression of $\pib$ is not explicit, however thanks to \eqref{system} we can derive information on the shape of the distribution.
In the sequel we denote by $\lceil x\rceil=\min\{m\in\mathbb{Z}; m\ge x\}$ for $x\in \R$.
\begin{prop}
The invariant probability measure $\widebar{\pi}$ admits a unique maximum at $1$ if $\mu_1\ge \mu_2$ and at $n_1>1$ otherwise.
The value of $n_1$ can be explicitly computed in function of the model parameters.
\end{prop}
\begin{proof}
The study of the existence of maxima of the distribution $\pib$ and thus $\widebar{\mu}$ is equivalent to we study of $\frac{\mu_{n+1}}{\mu_n}-1=\xi(n)$ for $\xi$ the function defined by
$$
\xi(x)=  \frac{b}{(x+1)\bigl(d+c(x+1)+B(rB(x+1)-D)/C\bigr)}-1, \quad \forall x\in[1,+\infty).$$
The sign of $\xi(x)$ is given by the sign of the polynomial
\ben
\label{poly}
\alpha x^2+\beta x+\gamma=0
\een
with
\bea
\alpha&= c+\frac{rB^2}{C}\\
\beta &=d+2c+2\frac{rB^2}{C}-\frac{BD}{C}\\
\gamma &=-b+d+c+\frac{rB^2}{C}-\frac{BD}{C}.
\eea
Its discriminant equals
$$
\beta^2-4\alpha\gamma = (d-\frac{BD}{C})^2+4(c+\frac{rB^2}{C})b,
$$ which is always positive.
Therefore the polynomial \eqref{poly} admits 2 real roots:
$$x_0=\frac{-\beta -\sqrt{(\beta^2-4\alpha\gamma)}}{2\alpha}\quad\text{and}\quad x_1=\frac{-\beta +\sqrt{(\beta^2-4\alpha\gamma)}}{2\alpha}.$$
The smallest root $x_0$ is always negative.
If $x_1>1$ then the invariant distribution $\widebar{\mu}$ admits exactly one mode at $n_1=\lceil x_1 \rceil$. Otherwise, the sequence $(\widebar{\mu}_n)_{n\ge1}$ is decreasing. 
To conclude the proof it remains to remark that the condition $x_1>1$ is equivalent to $\xi(1)>0$ which is the condition given in the proposition.
\end{proof}

\subsubsection{Numerics}
For each $\varepsilon\in(0,1]$ we proved in Section \ref{sec:ergodicity} that there exists a unique invariant probability measure $\pi^{\varepsilon}$ for the process $\Ze$. 
In this section, we study with numerical simulations the behavior of the sequence of invariant probability measures $(\pie)_{\varepsilon}$ as $\varepsilon\to0$.
An approximation of the invariant measure $\pi^{\varepsilon}$ is obtained by simulating 3000 times the prey-predator process $\Ze$ on a long time interval. One interest of the process $\Ze$, is that the flow for the fast predator population admits an explicit formula \eqref{phi_n}. Therefore, we are able to compute exact simulations of the process $\Ze$ for every $\varepsilon \in (0,1]$. The code for these simulations is available in \cite{costa:tel-01235792} (Chapter 3, Appendix A).\\
In the simulations, the demographic parameters are given by
\be
b=0.4, \quad d=0,\quad c=0.005, \quad B= 0.02, \quad r=2,\quad D=0\text{ and }\quad C=0.04.
\ee
In Figure \ref{fig:histo}, we draw three-dimensional histograms of the distributions of the invariant measure $\pi^{\varepsilon}$ for different values of $\varepsilon$ (in \subref{fig:histo1} $\varepsilon =1$, in \subref{fig:histo2}  $\varepsilon=0.1$ and in \subref{fig:histo3} $\varepsilon=0.00001$).
We observe that the support of these measures concentrates as $\varepsilon\to0$ on the set $\{(n,h^*_n),n\in\N^*\}$. This set corresponds to the support of the stationary distribution $\pib$.
\medskip\\
 \begin{figure}[h!]
 \begin{center}
 \begin{minipage}{0.3\linewidth}
  \subfloat[][\hspace{0.1cm}$\varepsilon=1$]{\includegraphics[scale=0.35]{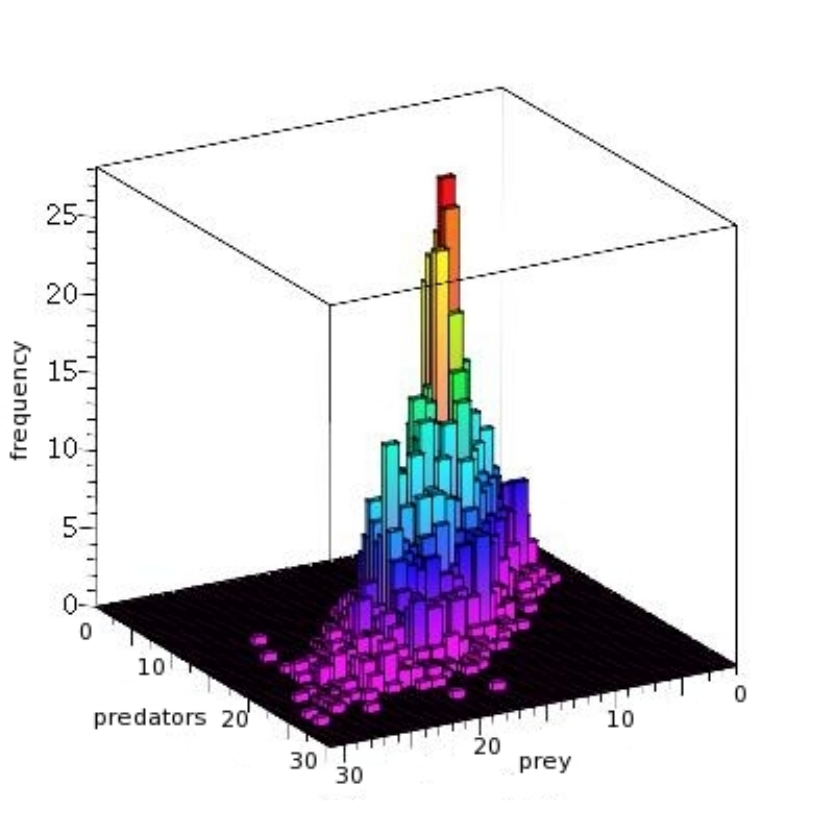}
 \label{fig:histo1}}
 \end{minipage}
 \hfill
 \begin{minipage}{0.3\linewidth}
  \subfloat[][\hspace{0.1cm} $\varepsilon=0.01$]{\includegraphics[scale=0.35]{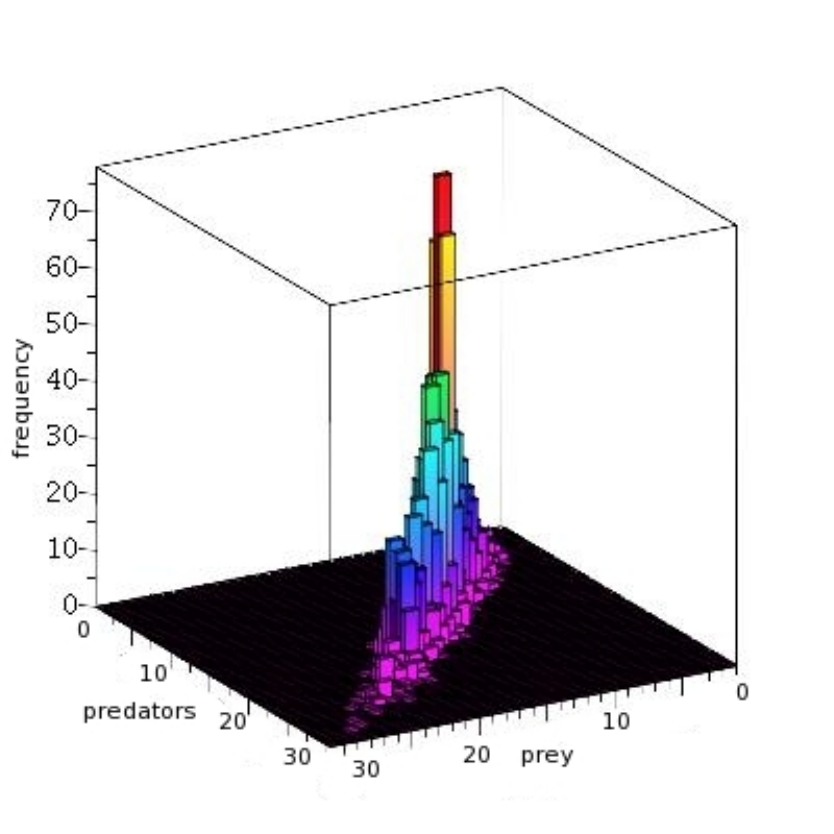}
 \label{fig:histo2}}
 \end{minipage}
 \hfill
 \begin{minipage}{0.3\linewidth}
  \subfloat[][\hspace{0.1cm}$\varepsilon=0.00001$]{\includegraphics[scale=0.35]{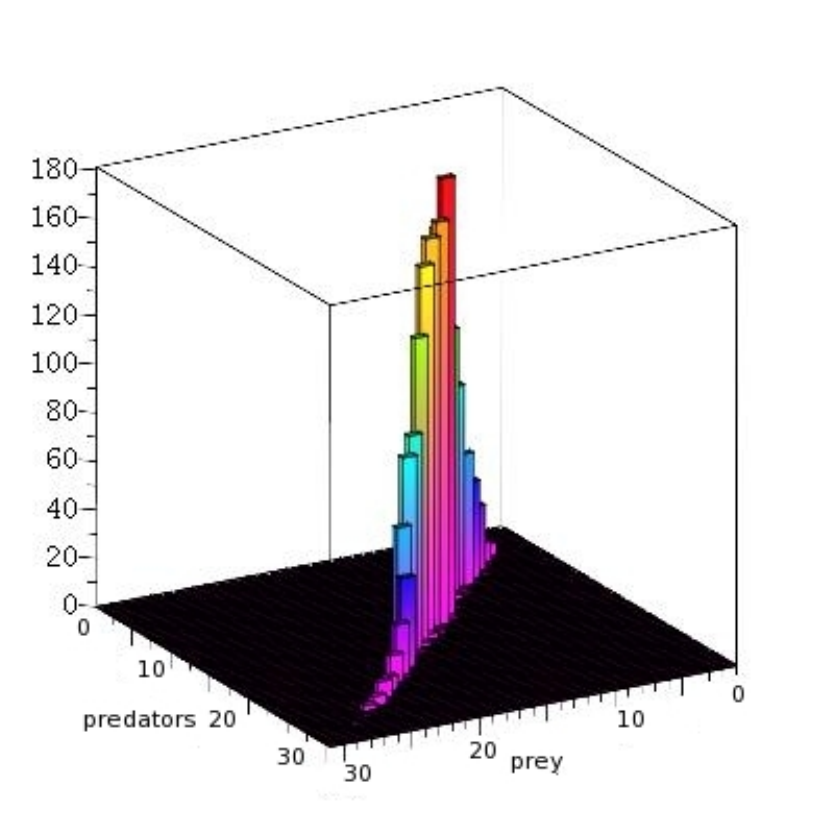}
 \label{fig:histo3}}
 \end{minipage}
 \caption[Approximation of the invariant measure $\pie$ for different values of $\varepsilon$]{\small{Approximation of the invariant measure $\pie$ for different values of $\varepsilon$.
  These histograms are built from 3000 iterations of the community process $\Ze$ until time $1000$. The parameters are $b=0.4$, $d=0$,  $c=0.005$, $B= 0.02$, $r=2$, $D=0$ and  $C=0.04.$}}
 \label{fig:histo}
 \end{center}
 \end{figure} 
\noindent We now compare these measures $\pie$ with the measure $\pib$. With the parameters of the simulations, the averaged invariant measure $\widebar{\pi}=\sum_{n=1}^{\infty} \mu_n \delta_{(n,h^*_n)}$ satisfies
\bea
& \mu_n =\frac{b^{n-1} }{(\widetilde{c})^{n-1} n(n!) }\mu_1, \quad \forall n\ge2\\
&\sum_{n=1}^{\infty} \mu_n=1,
\eea
where $\widetilde{c} = c+ rB^2/C$ is the apparent competition.\\
In the simulation, we approximate $\mu_1$ by
$$\mu_1\simeq\frac{1}{\sum_{k=1}^{50} (b/\widetilde{c})^{k-1}k(k)!}\simeq 2,69\cdot10^{-5}.$$
\noindent In Figure \ref{fig:histo-proj} we consider the marginal distribution of prey and predators given by the previous simulations ($\varepsilon =1$ in the left column,  $\varepsilon=0.1$ the middle column and $\varepsilon=0.00001$ in the right column). These distributions are projections of the histograms in Figure \ref{fig:histo}. We compare them with the projections of the averaged distribution $\widebar{\pi}$ represented with a black line joining the points $(n,\mu_n)$.
For these histograms, we chose subdivisions centered in the integers for the prey population and in the $(h^*_n)_{\ge1}$ for the predator populations. 
\\
We observe a rapid convergence of the marginal distributions of prey and of predators, toward the marginal averaged distributions.
\begin{figure}[h!]
 \begin{tabular}{>{\centering\arraybackslash}m{.01\textwidth} >{\centering\arraybackslash}m{.3\textwidth}>{\centering\arraybackslash}m{.3\textwidth}>{\centering\arraybackslash}m{.3\textwidth}}
 &$\varepsilon=1$&$\varepsilon=0.1$&$\varepsilon=0.00001$\\
 \begin{turn}{90}preys\end{turn}&
  \subfloat[][]{\includegraphics[scale=0.24]{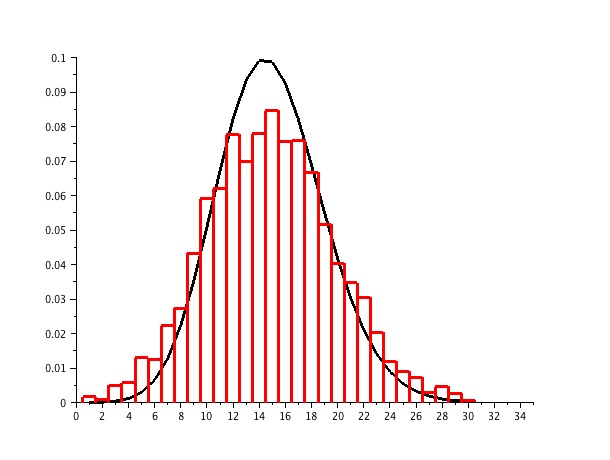}
 \label{fig:proie1}}
 &
  \subfloat[][]{\includegraphics[scale=0.24]{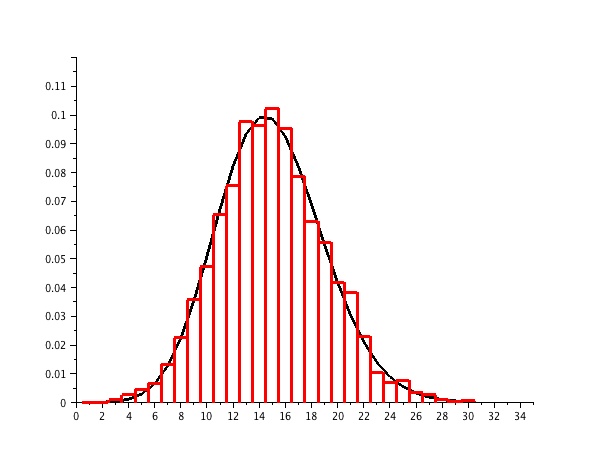}
 \label{fig:proie2}}
 &
  \subfloat[][]{\includegraphics[scale=0.24]{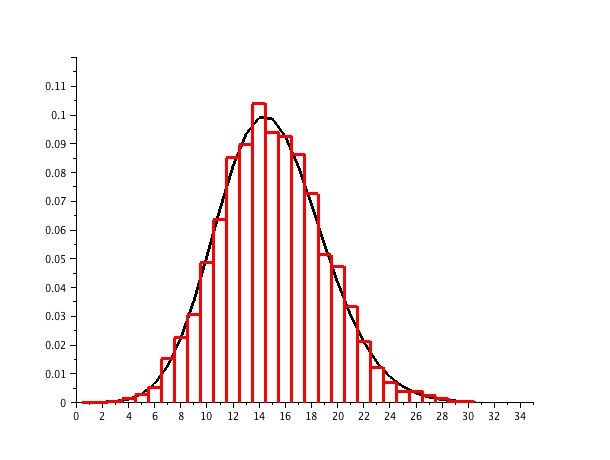}
 \label{fig:proie3}}\\
 \begin{turn}{90}predators\end{turn}
 &
  \subfloat[][]{\includegraphics[scale=0.24]{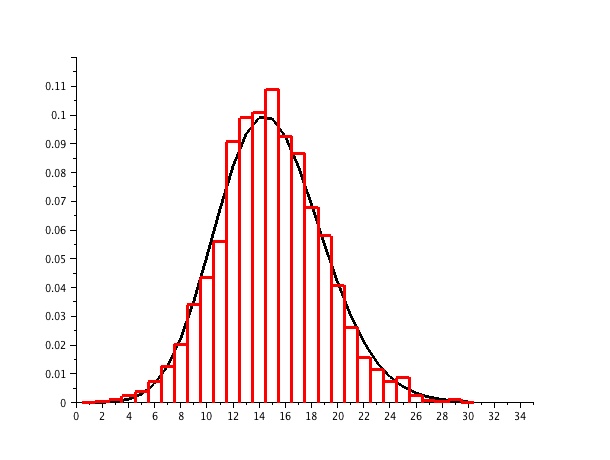}
 \label{fig:pred1}}
 &
  \subfloat[][]{\includegraphics[scale=0.24]{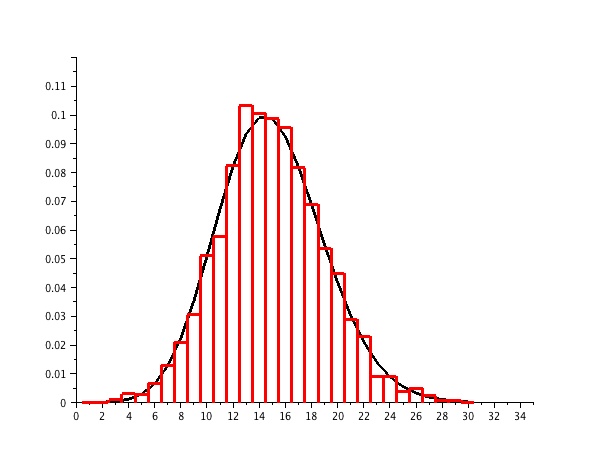}
 \label{fig:pred2}}
 &
  \subfloat[][]{\includegraphics[scale=0.24]{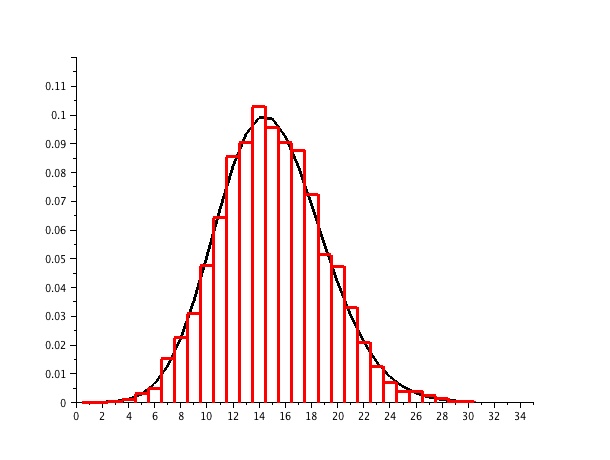}
 \label{fig:pred3}}
 \end{tabular}
 \caption[Approximation of the prey and the predator marginal invariant distributions for different values of $\varepsilon$]{\small{Marginal invariant distributions of the prey population in \subref{fig:proie1}-\subref{fig:proie2}-\subref{fig:proie3} and of the predator population \subref{fig:pred1}-\subref{fig:pred2}-\subref{fig:pred3} for different values of $\varepsilon$ ($\varepsilon =1$ in \subref{fig:proie1}-\subref{fig:pred1},  $\varepsilon=0.1$ in \subref{fig:proie2}-\subref{fig:pred2} and $\varepsilon=0.00001$ in \subref{fig:proie3}-\subref{fig:pred3}). These histograms are built from 3000 iterations of the community process $\Ze$ until time $1000$. The parameters are $b=0.4$, $d=0$,  $c=0.005$, $B= 0.02$, $r=2$, $D=0$ and $C=0.04.$}}
 \label{fig:histo-proj}
 \end{figure} 
\medskip

In these simulations, the sequence $(\pie)_{\varepsilon}$ seems to converge in law to $\pie$ but we have yet no mathematical proof of this result.

\section{Discussion}
\label{sec:discuss}
We introduced new models for prey-predator communities in which the predator dynamics is faster than the prey one. These stochastic models derive from a microscopic model of the community under the successive scalings of large predator population size and small predator mass. In both cases, we assume that the prey population size remains finite. This assumption corresponds to a biological reality as in the case of forest or experimental settings: for example in \cite{Umea} the authors study trees-insects communities in which the number of trees is of the order of a hundred. These models can be easily simulated and their dynamics is simple therefore they represent an alternative to slow-fast Lotka-Volterra dynamical systems \cite{RINALDI1992287}.\\

From an ecological point of view, we proved that our prey-predator process admits a long time distribution. Since the convergence to the invariant measure is exponentially fast, this invariant measure can be rapidly simulated using trajectories of the process. This distribution gives insight on the composition of prey-predator systems and can be used to compute the composition of trees-insects communities when the counting of insects may be difficult to do experimentally.
Moreover, in the situation where the mass ratio between prey individuals and predators is small, the limiting distribution admits an explicit expression, which is valuable for applications.\\

In this article, our approach was to consider first a large predator population size and then the small size of predators. The interest of these successive scalings was to introduce two different models corresponding to two different biological configurations. It would also be possible, and biologically relevant, to consider both scalings simultaneously by considering accelerated birth and death events as in \cite{champagnat2008individual} (Section 4.2). In this case other processes would arise as the coupling of a birth death process and a diffusion. The study of such processes will be considered in future works.\\

Knowing the long time behavior the prey-predator community is crucial to consider the evolutionary dynamics of the community. Indeed if we are interested in the phenotypic evolution of phenotype traits of prey and/or predators (as the development of specific tree defenses or of specialist insects strategies \cite{Umea}), one might be interested to consider the arrival of rare mutations as in adaptive dynamics settings (e.g. \cite{metz1996adaptive, DieckmannLaw96}). In this setting the stationary distribution represents the state of the resident community as a mutant arises. Therefore its knowledge is important to study the possible invasion of the mutant phenotype and to characterize the favourable strategies for prey or predators \cite{ChampagnatLambert07}.\\

Finally, even if our motivation was to describe the dynamics of ecological communities such as trees-insects communities, these models or methods could also be applied to very different questions such as epidemiology. For example, the study of the propagation of insect transmitted diseases such as Malaria impose to consider the interaction between insect and human populations \cite{aron1982population}. 
In this case the dynamics of insects is much faster than the human dynamics. Therefore piecewise deterministic processes could be introduce to model the dynamics of the disease and to study its outbreaks.

\section*{Acknowledgments}
I fully thank Sylvie M\'el\'eard for her continual guidance and her multiple suggestions on earlier versions of this article. I would also like to thank Gersende Fort, Carl Graham and Ga\"el Raoul for fruitful discussions during my work. I am grateful to the anonymous reviewers and associate editor whose comments improved the quality of this manuscript.\\
This article benefited from the support of the Chair ``Mod\'elisation Math\'ematique et Biodiversit\'e" of Veolia Environnement - \'Ecole Polytechnique - Museum National d'Histoire Naturelle - Fondation X.

{\setlength{\baselineskip}{0.9\baselineskip}
\bibliographystyle{plain} 

\begin{thebibliography}{10}

\bibitem{aldous1978stopping}
D.~Aldous.
\newblock Stopping times and tightness.
\newblock {\em The Annals of Probability}, 6(2):335--340, 1978.

\bibitem{allen2010introduction}
L.~J.S. Allen.
\newblock {\em An introduction to stochastic processes with applications to
  biology}.
\newblock CRC Press, 2010.

\bibitem{Amstrong80}
R.A. Armstrong and R.~McGehee.
\newblock Competitive exclusion.
\newblock {\em The American Naturalist}, 115(2):151--170, 1980.

\bibitem{aron1982population}
Joan~L Aron and Robert~M May.
\newblock The population dynamics of malaria.
\newblock In {\em The population dynamics of infectious diseases: theory and
  applications}, pages 139--179. Springer, 1982.

\bibitem{austin2008emergence}
T.~D. Austin.
\newblock The emergence of the deterministic hodgkin-huxley equations as a
  limit from the underlying stochastic ion-channel mechanism.
\newblock {\em The Annals of Applied Probability}, 18(4):1279--1325, 2008.

\bibitem{azema1967mesure}
J.~Azema, M.~Kaplan-Duflo, and D.~Revuz.
\newblock Mesure invariante sur les classes r{\'e}currentes des processus de
  markov.
\newblock {\em Probability Theory and Related Fields}, 8(3):157--181, 1967.

\bibitem{bayer2011existence}
C.~Bayer and K.~W{\"a}lde.
\newblock Existence, uniqueness and stability of invariant distributions in
  continuous-time stochastic models.
\newblock {\em Gutenberg School of Management and Economics: Discussion Paper
  Series}, 2011.

\bibitem{benaim2015qualitative}
M. Bena{\"\i}m, S. Le~Borgne, F. Malrieu, P-A.
  Zitt.
\newblock Qualitative properties of certain piecewise deterministic markov
  processes.
\newblock In {\em Annales de l'Institut Henri Poincar{\'e}, Probabilit{\'e}s et
  Statistiques}, volume~51, pages 1040--1075. Institut Henri Poincar{\'e},
  2015.

\bibitem{bouguet2013quantitative}
F. Bouguet.
\newblock Quantitative speeds of convergence for exposure to food contaminants.
\newblock {\em ESAIM: Probability and Statistics}, 19:482--501, 2015.

\bibitem{brown2004toward}
J.~H. Brown, J.~F. Gillooly, A.~P. Allen, V.~M. Savage, and G.~B. West.
\newblock Toward a metabolic theory of ecology.
\newblock {\em Ecology}, 85(7):1771--1789, 2004.

\bibitem{campillo2011stochastic}
F. Campillo, M. Joannides, and I. Larramendy-Valverde.
\newblock Stochastic modeling of the chemostat.
\newblock {\em Ecological Modelling}, 222(15):2676--2689, 2011.

\bibitem{CFM08}
N.~Champagnat, R.~Ferri{\`e}re, and S.~M{\'e}l{\'e}ard.
\newblock Unifying evolutionary dynamics: from individual stochastic processes
  to macroscopic models.
\newblock {\em Theoretical population biology}, 69(3):297--321, 2006.

\bibitem{ChampagnatLambert07}
N.~Champagnat and A.~Lambert.
\newblock Evolution of discrete populations and the canonical diffusion of
  adaptive dynamics.
\newblock {\em The Annals of Applied Probability}, 17(1):102--155, 2007.

\bibitem{champagnat2008individual}
N. Champagnat, R. Ferri{\`e}re, and S. M{\'e}l{\'e}ard.
\newblock From individual stochastic processes to macroscopic models in
  adaptive evolution.
\newblock {\em Stochastic Models}, 24(S1):2--44, 2008.

\bibitem{collet2013stochastic}
P.~Collet, S.~Martinez, S.~M{\'e}l{\'e}ard, and J.~San~Mart{\'\i}n.
\newblock Stochastic models for a chemostat and long-time behavior.
\newblock {\em Advances in Applied Probability}, 45(3):822--837, 2013.

\bibitem{costa2014}
M.~Costa, C.~Hauzy, N.~Loeuille, and S.~M\'el\'eard.
\newblock Stochastic eco-evolutionary model of a prey-predator community.
\newblock {\em Journal of Mathematical Biology}, 72(3):573--622, 2015.

\bibitem{costa:tel-01235792}
M. Costa.
\newblock {\em {Probabilistic and eco-evolutionary models for prey-predator communities}}.
  \newblock (\url{https://hal.archives-ouvertes.fr/tel-01235792/document}).
\newblock Theses, {Ecole Polytechnique}, September 2015.

\bibitem{costa2008stability}
O.~L.~V. Costa and F.~Dufour.
\newblock Stability and ergodicity of piecewise deterministic markov processes.
\newblock {\em SIAM Journal on Control and Optimization}, 47(2):1053--1077,
  2008.

\bibitem{crudu2012convergence}
A.~Crudu, A.~Debussche, A.~Muller, and O.~Radulescu.
\newblock Convergence of stochastic gene networks to hybrid piecewise
  deterministic processes.
\newblock {\em The Annals of Applied Probability}, 22(5):1822--1859, 2012.

\bibitem{crump1979some}
K.~S Crump and W-S.~C O'Young.
\newblock Some stochastic features of bacterial constant growth apparatus.
\newblock {\em Bulletin of Mathematical Biology}, 41(1):53--66, 1979.

\bibitem{Damuth81}
J.~Damuth.
\newblock Population density and body size in mammals.
\newblock {\em Nature}, 290:699--700, 1981.

\bibitem{davis1984piecewise}
M.~H.~A. Davis.
\newblock Piecewise-deterministic markov processes: A general class of
  non-diffusion stochastic models.
\newblock {\em Journal of the Royal Statistical Society. Series B.
  Methodological}, 46(3):353--388, 1984.

\bibitem{davis1993markov}
M.~H.~A. Davis.
\newblock {\em Markov Models \& Optimization}, volume~49.
\newblock CRC Press, 1993.

\bibitem{DieckmannLaw96}
U.~Dieckmann and R.~Law.
\newblock The dynamical theory of coevolution: a derivation from stochastic
  ecological processes.
\newblock {\em Journal of mathematical biology}, 34(5-6):579--612, 1996.

\bibitem{dufour1999stability}
F.~Dufour and O.~L.V. Costa.
\newblock Stability of piecewise-deterministic markov processes.
\newblock {\em SIAM Journal on Control and Optimization}, 37(5):1483--1502,
  1999.

\bibitem{Ethier&Kurtz}
N.~Ethier and T.Q. Kurtz.
\newblock {\em Markov Processes Characterization and Convergence}.
\newblock 1986.

\bibitem{fontbona2012quantitative}
J. Fontbona, H. Gu{\'e}rin, F. Malrieu.
\newblock Quantitative estimates for the long-time behavior of an ergodic
  variant of the telegraph process.
\newblock {\em Advances in Applied Probability}, 44(4):977--994, 2012.

\bibitem{FM04}
N.~Fournier and S.~M{\'e}l{\'e}ard.
\newblock A microscopic probabilistic description of a locally regulated
  population and macroscopic approximations.
\newblock {\em The Annals of Applied Probability}, 14(4):1880--1919, 2004.

\bibitem{genadot2014multi}
A.~Genadot.
\newblock A multi-scale study of a class of hybrid predator-prey models.
\newblock {\em arXiv preprint arXiv:1409.0376}, 2014.

\bibitem{grigorescu2014critical}
I. Grigorescu and M. Kang.
\newblock Critical scale for a continuous aimd model.
\newblock {\em Stochastic Models}, 30(3):319--343, 2014.

\bibitem{kurtz1992averaging}
T.~G. Kurtz.
\newblock Averaging for martingale problems and stochastic approximation.
\newblock In {\em Applied Stochastic Analysis}, pages 186--209. Springer, 1992.

\bibitem{kurtz2011equivalence}
T.~G. Kurtz.
\newblock Equivalence of stochastic equations and martingale problems.
\newblock In {\em Stochastic Analysis 2010}, pages 113--130. Springer, 2011.

\bibitem{last2004ergodicity}
G.~Last.
\newblock Ergodicity properties of stress release, repairable system and
  workload models.
\newblock {\em Advances in Applied Probability}, 36(2):471--498, 2004.

\bibitem{lotka}
A.~J. Lotka.
\newblock Elements of physical biology.
\newblock 1925.

\bibitem{ludwig1978qualitative}
D.~Ludwig, D.~D. Jones, and C.~S. Holling.
\newblock Qualitative analysis of insect outbreak systems: the spruce budworm
  and forest.
\newblock {\em The Journal of Animal Ecology}, 47(1):315--332, 1978.

\bibitem{metz1996adaptive}
J.~A.J. Metz, S.~A.H. Geritz, G.~Mesz{\'e}na, F.~J.A. Jacobs, and J.S.
  Van~Heerwaarden.
\newblock Adaptive dynamics, a geometrical study of the consequences of nearly
  faithful reproduction.
\newblock {\em Stochastic and spatial structures of dynamical systems},
  45:183--231, 1996.

\bibitem{meyntweedie1}
S.~P. Meyn and R.~L. Tweedie.
\newblock Stability of markovian processes i: Criteria for discrete-time
  chains.
\newblock {\em Advances in Applied Probability}, 24(3):542--574, 1992.

\bibitem{meyntweedie2}
S.~P. Meyn and R.~L. Tweedie.
\newblock Stability of markovian processes ii: Continuous-time processes and
  sampled chains.
\newblock {\em Advances in Applied Probability}, 25(3):487--517, 1993.

\bibitem{meyntweedie3}
S.~P. Meyn and R.~L. Tweedie.
\newblock Stability of markovian processes iii: Foster-lyapunov criteria for
  continuous-time processes.
\newblock {\em Advances in Applied Probability}, 25(3):518--548, 1993.

\bibitem{meyntweedie}
S.~P. Meyn and R.~L. Tweedie.
\newblock {\em Markov Chains and Stochastic Stability}.
\newblock Cambridge Mathematical Library. Cambridge University Press, 2009.

\bibitem{RINALDI1992287}
S.~Rinaldi and S.~Muratori.
\newblock Slow-fast limit cycles in predator-prey models.
\newblock {\em Ecological Modelling}, 61(3):287 -- 308, 1992.

\bibitem{Umea}
K.~M. Robinson, P.~K. Ingvarsson, S.~Jansson, and B.~R. Albrectsen.
\newblock Genetic variation in functional traits influences arthropod community
  composition in aspen (\textit{Populus tremula} l.).
\newblock {\em PLoS ONE}, 7(5):e37679, 05 2012.

\bibitem{rudnicki2007influence}
R.~Rudnicki and K.~Pich{\'o}r.
\newblock Influence of stochastic perturbation on prey--predator systems.
\newblock {\em Mathematical biosciences}, 206(1):108--119, 2007.

\bibitem{tweedie1994topological}
R.~L. Tweedie.
\newblock Topological conditions enabling use of harris methods in discrete and
  continuous time.
\newblock {\em Acta Applicandae Mathematica}, 34(1-2):175--188, 1994.

\bibitem{volterra}
V.~Volterra.
\newblock Fluctuations in the abundance of a species considered mathematically.
\newblock {\em Nature}, 118:558--560, 1926.

\end{thebibliography}

\par}

\end{document}